\documentclass[final,onefignum,onetabnum]{siamart190516}

\usepackage{amsfonts}
\usepackage{graphicx}
\usepackage{epstopdf}
\usepackage{algorithmic}
\usepackage{subcaption}

\usepackage[sort,nocompress]{cite}

\usepackage{pgffor}
\usepackage{tikz}
\usepackage{pgfplots}
\usepackage{pgfplotstable}
\pgfplotsset{compat=1.14}

\usepackage{enumitem}
\setlist{leftmargin=*}

\usepackage{cleveref}

\crefformat{chapter}{\S#2#1#3}
\crefmultiformat{chapter}{\S\S#2#1#3}{and~#2#1#3}{, #2#1#3}{, and~#2#1#3}

\crefformat{section}{\S#2#1#3}
\crefmultiformat{section}{\S\S#2#1#3}{and~#2#1#3}{, #2#1#3}{, and~#2#1#3}

\crefformat{subsection}{\S#2#1#3}
\crefmultiformat{subsection}{\S\S#2#1#3}{and~#2#1#3}{, #2#1#3}{, and~#2#1#3}

\crefformat{subsubsection}{\S#2#1#3}
\crefmultiformat{subsubsection}{\S\S#2#1#3}{and~#2#1#3}{, #2#1#3}{, and~#2#1#3}

\ifpdf
  \DeclareGraphicsExtensions{.eps,.pdf,.png,.jpg}
\else
  \DeclareGraphicsExtensions{.eps}
\fi

\newsiamremark{remark}{Remark}
\newsiamremark{hypothesis}{Hypothesis}
\crefname{hypothesis}{Hypothesis}{Hypotheses}
\newsiamthm{claim}{Claim}

\newcommand{\mute}[1]{}

\colorlet{RED}{red}

\headers{Geometry of Graph Partitions via Optimal Transport}{Abrishami, Guillen, Rule, Schutzman, Solomon, Weighill, and Wu}

\title{Geometry of Graph Partitions via Optimal Transport\thanks{\mute{Submitted to the editors October, 2019. }This work was initiated during the summer 2019 Voting Rights Data Institute in a team project on distances between partitions (faculty lead: Justin Solomon); authors are in alphabetical order by last name.  Additional team members are listed in the acknowledgements.  The corresponding author is Tara Abrishami.
}}

\author{
Tara Abrishami%
\thanks{Princeton University, Princeton, NJ (\email{taraa@princeton.edu})}%
\and
Nestor Guillen%
\thanks{Texas State University, San Marcos, TX (\email{nestor@txstate.edu})}%
\and
Parker Rule%
\thanks{Tufts University, Medford, MA (\email{parker.rule@tufts.edu})}%
\and
Zachary Schutzman%
\thanks{University of Pennsylvania, Philadelphia, PA (\email{ianzach@seas.upenn.edu})}%
\and
Justin Solomon%
\thanks{Massachusetts Institute of Technology, Cambridge, MA (\email{jsolomon@mit.edu})}%
\and
Thomas Weighill%
\thanks{Tufts University, Medford, MA (\email{thomas.weighill@tufts.edu})}%
\and
Si Wu%
\thanks{Northeastern University, Boston, MA (\email{wu.si1@husky.neu.edu})}%
}

\usepackage{amsopn}

\usepackage{bbm}

\newcommand{\R}{\mathbb{R}}
\newcommand{\1}{\mathbbm{1}}
\newcommand{\Prob}{\text{Prob}}

\newcommand{\M}{\text{M}}
\newcommand{\DS}{\mathrm{DS}}
\newcommand{\st}{\textrm{subject to}}
\newcommand{\diam}{\text{diam}}

\DeclareMathOperator*{\argmin}{argmin}
\begin{document}

\maketitle

\begin{abstract}
We define a distance metric between partitions of a graph using machinery from optimal transport.  Our metric is built from a linear assignment problem that matches partition components, with assignment cost proportional to transport distance over graph edges.  We show that our distance can be computed using a single linear program without precomputing pairwise assignment costs and derive several theoretical properties of the metric. 
Finally, we provide experiments demonstrating these properties empirically, specifically focusing on its value for new problems in ensemble-based analysis of political districting plans.
\end{abstract}



\section{Introduction}

Several mathematical and computational problems involve collections of graph partitions with fixed numbers of components.  Example application scenarios include tracking and clustering of evolving communities in a network, as well as analysis of political redistricting plan ensembles---an application we will study in detail below.  Because specifying a single partition requires a label for every vertex, 
however, it can be difficult to visualize and navigate such a collection.  Additionally, because the number of possible partitions typically exponentiates in the size of the underlying graph, collections of partitions usually are extremely large.

Enriching the set of partitions with a geometric structure provides a means of understanding the vast space 
of partitions. In particular, a distance between partitions allows us to quantify the qualitative notion of similarity or difference between two partitions and can provide insight into the structure of the space of partitions as a whole.  
For instance, a large family of partitions might be close in a given metric to elements in a smaller one; in this case, we can efficiently infer information about the larger collection from a small representative subsample. 
More broadly, a distance metric can evaluate whether a given sample of partitions spreads over the set of partitions or concentrates in a smaller region. A metric also yields a visualization tool: Given a finite family of partitions, one may compute pairwise distances and use them with an embedding algorithm  
to create a two- or three-dimensional Euclidean visualization.

Motivated by the challenges above, we present a distance on the space of graph partitions motivated by optimal transport.  Our model uses transport to measure pairwise relationships between the components of two partitions; a linear assignment problem then extracts the minimum cost (perfect) matching between partition components.  Our formulation can be understood as a \emph{hierarchical} transport problem that is invariant to the ordering of the individual components and sensitive to geometry, in contrast to simpler overlap-based measures, e.g., those based on KL-divergence or total variation.  We derive some theoretical properties of our distance and provide an extension to unbalanced problems where the partition components are weighted unequally.

Our target application is in \emph{political redistricting}, where we can use these tools to compare districting plans for some geographic region.  In recent ensemble-based approaches to districting plan analysis, a large collection of feasible voting districts is generated computationally as a baseline for the evaluation of a proposed plan; the baseline set samples the achievable properties for plans given the political geography of a state.  An issue in current ensemble-based redistricting pipelines is that the only two options for visualizing and navigating the ensemble are (1) showing a few randomly selected plans as examples or (2) plotting the empirical distribution of the values of a given measure, such as the number of districts won by a particular political party or a compactness metric summarizing the shapes of the districts, over the ensemble.  The first option shows an exceptionally small subset of ensembles that can number in the millions, while the second is an indirect means of understanding the relationships between plans.  Here, we show that our transport metric---coupled with embedding methods like multidimensional scaling (MDS)---provides an alternative to these two options, giving a direct and intuitive means of visualizing an entire ensemble.  Our experiments confirm the value of this approach in practice, applied to both synthetic and real-world datasets.

\paragraph*{Outline} 
In \cref{section:Related work}, we comment on some of the most relevant literature, including that from optimal transport and ensemble-based redistricting. Next, \cref{section:Preliminaries} reviews the basic terms and notation. \cref{section:basic optimal transport} reviews the basics of optimal transport and the Wasserstein metric, including  the Kantorovich and Beckmann problems. In \cref{section:distance mass balanced partitions} and \cref{section:unbalanced_distances}, we introduce the distance metric between partitions and prove some basic properties. In \cref{section:bounds}, we compare the transportation-based metric to other metrics over partitions. Finally, \cref{sec:experiments} illustrates uses of the metric with both simulated as well as real geographic data and  \cref{sec:conclusion}  summarizes our work, including open problems and avenues for future research. 

\section{Related Work}\label{section:Related work}

\paragraph*{Optimal transport} Optimal transport is the problem of computing a matching between supply and demand---represented as two measures on a geometric space---with minimal cost. This problem was originally posed by Monge \cite{monge1781memoire} but did not see much activity until Kantorovich's work more than a century later \cite{kantorovich1942translocation}; see \cite{galichon2018optimal,villani2003topics,santambrogio2015optimal} for thorough introductions.

While the optimal transport problem was initially concerned with matching probability measures, later work considered more general instances such as partial transportation \cite{caffarelli2010free} and matching general measures \cite{figalli2010transportation}.
The latter is an example of \emph{unbalanced} optimal transport, a popular topic in modern theoretical and applied transport 
\cite{lombardi2015eulerian,chizat2018scaling}.

Spurred by applications in machine learning, computer vision, and other disciplines, a wealth of computational techniques has become available for approximating transport distances and derived quantities; see \cite{peyre2019computational} for a recent survey. Particularly relevant to our work is transport over graphs with shortest-path distance as the cost.  This problem is known in the computational literature as \emph{minimum-cost flow without edge capacities} \cite{ahuja1988network} and in the transport literature as the 1-Wasserstein distance or Beckmann problem \cite{beckmann1952continuous,santambrogio2013prescribed}.  See \cite{essid2018quadratically} for a survey of computational methods for this problem.

\paragraph*{Geometry of partitions}
Motivated by the study of partitions of minimal perimeter, Leonardi and Tamanini  introduced a metric on the space of (measurable) partitions of subset of Euclidean space \cite{leonardi2002metric}. Their distance is based on the measure of the symmetric difference between sets, producing a complete and separable metric space but only weakly using the geometry of the underlying domain. 
There is additionally a literature on the geometry of set-theoretic partitions, i.e.\ decompositions of a ground set into disjoint subsets \cite{day1981complexity,meilua2007comparing,denoeud2006comparison}, which examines the properties of metrics on the lattice on these set partitions. 

\paragraph*{Ensemble-based redistricting}

A growing body of scientific research centers on ensemble-based redistricting, which uses mathematical and algorithmic techniques to generate and analyze a collection of thousands or millions of candidate districting plans that meet some criteria \cite{chikina2017assessing,herschlag2017evaluating,Mattingly2018, deford2019redistricting, bangia2017redistricting, chen2013unintentional, chen2015cutting}. These methods and analyses are increasingly used as quantitative tools by legal experts, policymakers, and the public at-large to inform debates around redistricting. This circumstance is particularly salient in the legal context, where ensembles of voting maps generated with these methods are being submitted as evidence to state and federal courts in redistricting and voting rights cases, underlining the pressing need to understand quantitative and qualitative properties of ensembles. This is challenging because the space of maps meeting a reasonable collection of criteria is extremely large and poorly understood, making it difficult to quantify the ``diversity'' of a given collection of districting plans. Previous works focus on the distributions of several statistics of interest, such as the number of districts won by a particular political party or a compactness metric, and perform statistical analysis in this lower-dimensional space.

For such analyses to be robust, they should be performed on a sample of candidate districting plans representative of the universe of valid plans, such as those that meet legal criteria. If an ensemble contains many plans that are all slight variations of one another, a projection to summary statistics may be misleading, since those similar plans likely yield similar statistics, which would in turn be over-represented in the analysis. Assessing the diversity of a sample requires some measure of dissimilarity, and a rigorous development of such a measure is not present in the previous literature. \mute{\color{red} The existing ``geometry of redistricting'' literature largely focuses on the geometric properties of the shapes of the districts themselves} {This paper presents a novel direction within the ``geometry of redistricting'' that is somewhat orthogonal to the primary direction of the field, which largely focuses on analyzing the shapes of the districts themselves;} \mute{\color{red} This is separate from the task of embedding the districting plans themselves in space with geometric properties.}  for a survey of classical approaches in shape analysis for political districts, see \cite{young1988compact}.

\section{Preliminaries}\label{section:Preliminaries}

A graph will be denoted by $G = (V, E)$, with $V$ denoting the set of vertices and $E$ the set of edges. If the graph is weighted with a weight $\omega:E \to \mathbb{R}$ then we will write $G = (V,E,\omega)$. The signed incidence matrix associated to the graph $G$ will be denoted by $P$:
\begin{align*}
P_{ev}:= \left\{
\begin{array}{rl}
-1 & \textrm{ if }e=(v,w)\textrm{ for some }w\in V\\
1 & \textrm{ if }e=(w,v)\textrm{ for some }w\in V\\
0 & \textrm{ otherwise.}
\end{array}
\right.
\end{align*}

We denote by $\M(V)$ the set of all mass distributions over $V$ and by $\Prob(V) \subset \M(V)$ the set of all probability distributions over $V$. Specifically, 
\begin{align*} 
  \M(V) & = \{ x \in \R^{|V|} \mid x(v) \geq 0 \;\forall\;v\in V\},\textrm{ and}\\
  \Prob(V) & = \left\{x \in \R^{|V|} \mid x(v) \geq 0 \;\forall\;v\in V\textrm{ and }\sum_{v\in V} x(v) = 1\right\}.
\end{align*} 
We will also consider the set of mass distributions over the product $V\times V$
\begin{align*} 
  \M(V\times V) & = \{ x \in \R^{|V|^2} \mid x(v,w) \geq 0 \;\forall\;v,w\in V\}.
\end{align*}
Let $\Prob(V)^k$ be the set of $k$-tuples of elements of $\Prob(V)$, and let $\Prob(V)^{k*} \cong \Prob(V)^k/S_n$ be the set of $k$-tuples of elements of $\Prob(V)$ up to reordering. Similarly, let $M(V)^k$ be the set of $k$-tuples of elements of $M(V)$, and let $M(V)^{k*}\cong M(V)^k/S_n$ be the set of $k$-tuples of $M(V)$ up to reordering. \Cref{table:notation} provides relevant notation.

\begin{table}[H]
\centering
\begin{tabular}{l|l} 
\allowdisplaybreaks
Notation  & Definition \\
\hline

$G = (V,E)$ & graph with vertices $V$ and edges $E$\\
$G = (V,E,\omega)$ & weighted graph with vertices $V$, $E$, and weights $\omega:E\to \mathbb{R}$\\
$d(v,w)$ & Shortest-path distance between $v,w\in V$\\
$M(V)$ & set of all mass distributions on $V$ \\
$M(V)^k$ & ordered $k$-tuples of mass distributions on $V$ \\
$M(V)^{k*}$ & the set $M(V)^k$ modulo index rearrangements \\
$\Prob(V)$ & set of all probability distributions on $V$\\
$\Prob(V)^k$ & ordered $k$-tuples of probability distributions on $V$\\
$\Prob(V)^{k*}$ & the set $\Prob(V)^k$ modulo index rearrangements\\
$W_1(f,g)$ & Wasserstein distance between $f$ and $g$
\end{tabular}
\caption{Notation}\label{table:notation}
\end{table}

\section{Transport distances}\label{section:basic optimal transport}

In this section, we review some notions from the theory of optimal transport that will be relevant to our discussion.  We limit to a few basic results from transport over graph domains; see \cite{villani2003topics,santambrogio2015optimal} for the general case.

Let $x,y\in \M(V)$. A \emph{coupling} or \emph{transport plan} between $x$ and $y$ is a function $\pi:V\times V\to\mathbb{R}_+$ such that 
\begin{align*}
  \sum \limits_{w \in V} \pi(v,w) = x(v)\;\;\forall\;v \in V\ \textrm{and}\ 
  \sum \limits_{v\in V} \pi(v,w) = y(w)\;\;\forall\;w\in V.
\end{align*}
We will use $\Pi(x,y)$ to denote the set of such couplings. 

\begin{remark}
   If for $x,y\in \M(V)$ there is at least one $\pi \in \Pi(x,y)$, then
\begin{align*}
  \sum \limits_{v\in V} x(v) = \sum\limits_{v\in V}\sum\limits_{w\in V} \pi(v,w)\textrm{ and }
  \sum \limits_{w\in V} y(w) = \sum \limits_{w\in V}\sum\limits_{v\in V}\pi(v,w).
\end{align*}
The sums on the right-hand sides of the two equations are finite rearrangements of each other and therefore must agree. Conversely, if $x$ and $y$ have the same total mass $m > 0$, then the product distribution $\pi(v,w) = \tfrac{1}{m}x(v)y(w)$ belongs to $\Pi(x,y)$.
This means that there are admissible plans between $x$ and $y$ when (and only when) they have these two distributions have the some total mass.

\end{remark}

The \emph{total transportation cost} of a plan $\pi \in \Pi(x,y)$ is $\sum_{v,w \in V}d(v,w)\pi(v,w)$; this objective function states that the cost of moving mass between vertices $v,w\in V$ is the shortest-path distance $d(v,w)$. Then, the \emph{transport distance} (also known as 1-Wasserstein distance) between $x$ and $y$ is defined as the minimum total transportation cost for a plan $\pi \in \Pi(x,y)$, denoted $W_1(x,y)$:
\begin{align}\label{eqn:transport metric definition}
  W_1(x,y) := \min \limits_{\pi \in \Pi(x,y)} \sum\limits_{v,w \in V} d(v,w)\pi(v,w).
\end{align}
By convention, we take $W_1(x,y)= +\infty$ if $x$ and $y$ do not have the same total mass. A consequence of the more general theory of optimal transport is that $W_1$ defines a metric in $\Prob(V)$; see \cite[Chapter 7]{villani2003topics} or \cite[Chapter 5]{santambrogio2015optimal} for a general discussion or \cite{cuturi2014ground} for a proof specifically in the discrete case.  

In computational practice, formulation \eqref{eqn:transport metric definition} can be difficult to solve because it requires computing and operating on the set of $|V|\times |V|$ pairwise distances.  In graph theory, however, this problem is known as \emph{minimum cost flow without edge capacities} and admits an alternative formulation with a number of variables linear in $|E|$:
\begin{equation}\label{eq:w1graph}
    W_1(x,y)=\left\{
    \begin{array}{rl}
    \min\limits_{J\in \R^{|E|}} & \sum\limits_{e\in E}\omega(e)|J_e|\\
    \st & P^\top J = y-x,
    \end{array}
    \right.
\end{equation}
where $\omega$ denotes edge weights and $P$ is the incidence matrix associated to $G$. If the graph is unweighted, we can take $\omega(\cdot)\equiv1.$ The equivalence between \eqref{eqn:transport metric definition} and \eqref{eq:w1graph} is discussed in \cref{section:unbalanced transport p=1} in the broader setting of unbalanced transport. See also \cite{essid2018quadratically} and references therein for motivation for this formula as well as references to techniques that can solve the linear program associated to \eqref{eq:w1graph} in practice.

\section{The distance on partitions: balanced case}\label{section:distance mass balanced partitions}

In this section, we propose a distance between graph partitions that lifts the transport distances described above. This distance is defined in two steps: computing distances between partition components and subsequently finding a minimum-cost matching between the partition components. In particular, we take the distance between components to be the Wasserstein distance and use linear assignment to find the matching. With this definition in place, we prove some basic properties of the lifted distance and give a formulation as a single combined linear program rather than a two-step procedure.

\subsection{Distances between components}\label{subsec:balanced_comp_distances}
Let $G = (V, E)$ be a graph, and let $(V_1, \hdots, V_k)$ be a partition of the vertices of $G$. We can represent $(V_1, \hdots, V_k)$ by an element of $\Prob(V)^{k*}$ as follows: To every component $V_i$, we associate a vector $x_i \in \R^{|V|}$ such that 
\begin{equation}\label{eqn:uniform_rep} x_i(v) = \begin{cases} \frac{1}{|V_i|} &  \textrm{if }v \in V_i \\ 0 & \text{otherwise.} \end{cases} \end{equation}
Then, $X = (x_1, \hdots, x_k)\in\Prob(V)^{k*}$ gives a concrete representation of the partition $(V_1, \hdots, V_k)$.
This expression defines a balanced representation of partitions, because $\sum_{v \in V} x_i(v) = 1$ for all $i \in \{1, \hdots, k\}$. The case of unbalanced representations is covered in \cref{section:unbalanced_distances}. 

\begin{remark} \label{remark:normalized_representation} Given a strictly positive weight function on the vertices $\omega : V \to \R_+$, we can give an alternative definition of the vector $x_i$ associated to component $V_i$ as 
\begin{equation*}\label{eqn:weighted_rep} x_i(v) = \begin{cases} \frac{\omega(v)}{\sum\limits_{u \in V_i} \omega(u)} &  \textrm{if }v \in V_i \\ 0 & \text{otherwise.} \end{cases} \end{equation*} In our target application of political redistricting, this alternative definition can be useful when incorporating populations associated with census units. 
\end{remark}

Now that we have a representation of graph partitions in $\Prob(V)^{k*}$, we can define a distance between components of partitions. Let $X = (x_1, \hdots, x_k) \in \Prob(V)^{k*}$ and $Y = (y_1, \hdots, y_k) \in \Prob(V)^{k*}$ be partitions of $G$ using  the (unweighted) Wasserstein distance, defined in \cref{section:basic optimal transport}. In particular, for any $x_i \in X$ and $y_j \in Y$, we take 
\begin{align*}
\label{eqn:wasserstein-district}
  W_1(x_i,y_j)= \left\{
    \begin{array}{rl}
    \min\limits_{J\in \R^{|E|}} & \sum\limits_{e\in E}|J_e|\\
    \st & P^\top J = y_j-x_i.
    \end{array}\right.
\end{align*}

\subsection{Distances between partitions}\label{subsec:balanced_part_distances} In \cref{subsec:balanced_comp_distances}, we endowed the space of partition components with the Wasserstein distance. Here, we lift this distance to a distance between partitions using a linear assignment problem. 

For ease of notation, we define the relevant constraint set for our problem:
\begin{definition}[Birkhoff polytope]
The \emph{Birkhoff polytope} $\DS_k$ is the set of all $k\times k$ \emph{doubly stochastic} matrices, the nonnegative matrices whose rows and columns sum to 1:
\begin{equation}
    \DS_k = \{ S\in\R^{k\times k} \mid S\1=\1, S^\top \1= \1,\textrm{ and }S\geq0\}.
\end{equation}
\end{definition}

The Birkhoff--von Neumann theorem gives that $\DS_k$ is a convex polytope and its vertices are exactly the \textit{permutation matrices}, those elements of $\DS_k$ with integer entries.

\begin{definition}[Lifted distance]
Given a distance $C:\Prob(V)\times \Prob(V) \to \mathbb{R}$, the \emph{lifted distance} $A : \Prob(V)^{k*} \times \Prob(V)^{k*} \to \R$ between partitions $X$ and $Y$ is defined as  
\begin{equation}\label{eqn:A definition 1}
A(X,Y)=\left\{
\begin{array}{rl}
  \min\limits_{S\in \R^{k \times k}} & \sum\limits_{ij} S_{ij} C(x_i, y_j) \\
                 \st & S\in\DS_k,
\end{array}\right.
\end{equation}
where $X = (x_1, \hdots, x_k), Y = (y_1, \hdots, y_k)\in \Prob(V)^{k*}$ and $C(\cdot,\cdot)$ is a distance between partition components.  Unless otherwise noted, we will take $C=W_1$, the Wasserstein distance \eqref{eq:w1graph}.
\end{definition}
Because the extreme points of the feasible region for $S$ are permutation matrices, the minimizer of this linear program can be interpreted as a matching between the components of $X$ and the components of $Y$, and the distance is the sum of the pairwise distances between matched components. 

\begin{remark} 
Many properties of this lifting are independent of the ground metric, which in \eqref{eqn:A definition 1} is the transport distance $W_1$.  From a broader perspective, we can view this construction as an instance of \emph{hierarchical} optimal transport, i.e., a matching problem whose cost comes from another matching problem; see \cite{yurochkin2019hierarchical} for discussion of another example in natural language processing.
\end{remark}

Before studying properties of our construction, we verify that \eqref{eqn:A definition 1} lifts any distance between components (i.e. a distance on $\Prob(V)$) to a distance on $\Prob(V)^{k*}$.

\begin{proposition}
\label{prop:plan_metric}
Given any metric $C:\Prob(V)\times \Prob(V) \to \R$, the lifted distance $A: \Prob(V)^{k*} \times \Prob(V)^{k*} \to \R$ is a metric on $\Prob(V)^{k*}$. \end{proposition} 

\begin{proof}
 Since $C$ is a metric and $S$ is nonnegative, it is immediate that $A$ is nonnegative and symmetric. 

Let $X=(x_1, x_2, ..., x_k),Y=(y_1, y_2, ..., y_k)\in \Prob(V)^{k*}$, and suppose $X \equiv Y$ in $\Prob(V)^{k*}$. Then, there exists some permutation $P$ with 
$P_{ij} = 1$ if and only if $x_i = y_j$. Since $C(x_i, y_j) = 0$ when $x_i = y_j$, we have
$\sum_{ij} P_{ij} C(x_i, y_j) = 0.$  Hence, when $X \equiv Y$ in $\Prob(V)^{k*}$, $A(X, Y) = 0$.  
Conversely, suppose $A(X, Y) = 0$, and let $S$ minimize \eqref{eqn:A definition 1}. 
If $S_{ij} = 1$, then since the objective value is zero we must have $C(x_i, y_j) = 0$. Because $C$ is a metric, $C(x_i, y_j) = 0$ if and only if $x_i = y_j$. Therefore, 
$X \equiv Y$ in $\Prob(V)^{k*}$. 

Now, suppose $X,Y,Z\in \Prob(V)^{k*}$. Let $N$ and $W$ be minimizing permutation matrices with $A(X,Y) = \sum_{ij}N_{ij}C(x_i, y_j)$ and $A(Y,Z) = \sum_{jl}W_{jl}C(y_j, z_l)$. Then, 
\begin{align*} A(X,Y) + A(Y,Z) &=\sum_{ij} N_{ij}C(x_i,y_j) + \sum_{jl} W_{jl}C(y_j,z_l) \\
&= \sum_{i,l} \sum_{j} \left(N_{ij} C(x_i, y_j) + W_{jl} C(y_j, z_l)\right). \end{align*} 
Because $N$ and $W$ are permutations, for a fixed $j$, there is a unique $i_j$ such that $N_{i_jj} = 1$ and a unique $l_j$ such that $W_{jl_j} = 1$. Therefore, 
\begin{align*} A(X, Y) + A(Y, Z) = \sum_j \big(C(x_{i_j}, y_j) + C(y_j, z_{l_j})\big) 
\geq \sum_j C(x_{i_j}, z_{l_j}), \end{align*}  
by the triangle inequality. Let $B = NW$. Then, 
$\sum_j C(x_{i_j}, z_{l_j}) = \sum_{il} B_{il} C(x_i, z_l).$ Because $B$ is a permutation, $\sum_{il} B_{il} C(x_i, z_l) \geq A(X, Z)$. Therefore, $A(X, Y) + A(Y, Z) \geq A(X, Z),$ verifying a triangle inequality.
\end{proof}

\begin{remark}
Since $C$ is a metric, \cref{prop:plan_metric} likely follows from general results about discrete transport, e.g.\ \cite[Theorem 1]{cuturi2014ground}. We include the direct proof since metric properties follow directly from our definition.
\end{remark}

\subsection{Basic properties}\label{subsec:balanced_properties}

In this section, we prove several basic properties of the lifted distance. First, we show that if two partitions have a component in common, there exists an optimal matching that fixes the shared component:

\begin{proposition}
\label{prop:matched_districts}
Let $X = (x_1, \hdots, x_k)$ and $Y = (y_1, \hdots, y_k)$ be two partitions of $G$, and suppose $x_a = y_b$. Then, there exists a matching $S$ such that $S_{ab} = 1$ and $S$ is an optimizer for the lifted distance \eqref{eqn:A definition 1}.
\end{proposition}
\begin{proof} 

Let $P$ be a permutation matrix that is an optimizer for $A(X, Y)$, so $A(X, Y) = \sum_{ij} C(x_i,y_j) P_{ij}$. Suppose that $P$ maps $x_a$ to some $y_d$ and some $x_c$ to $y_b$. Let $S$ be the permutation obtained by matching $x_a$ to $y_b$, $x_c$ to $y_d$, and every other component in $X$ to its image under $P$. It is clear that
\begin{align*} \sum_{ij} S_{ij} C(x_i, y_i) &= \sum_{ij} P_{ij} C(x_i, y_j) - C(x_a, y_d) - C(x_c, y_b) + C(x_a, y_b) + C(x_c, y_d).
\end{align*}
Since $C$ is a metric, the triangle inequality yields $C(x_a, y_d) + C(x_c, y_b) \geq C(x_c, y_d).$ Because $x_a=y_b$ by assumption, $C(x_a, y_b) = 0$, so
\begin{align*}
- C(x_a, y_d) - C(x_c, y_d) + C(x_a, y_b) + C(x_c, y_d) \leq 0,
\end{align*}
and therefore
\begin{align*} \sum_{ij} S_{ij} C(x_i, y_i) \leq \sum_{ij} P_{ij} C(x_i, y_j).
\end{align*}
Since $P$ is an optimizer for the lifted distance, we also have that 
$$\sum_{ij} P_{ij} C(x_i, y_j) \leq \sum_{ij} S_{ij} C(x_i, y_j).$$
Therefore, $S$ is an optimizer for $A(X, Y)$ with $S_{ab} = 1$.   
\end{proof}

The current formulation of the lifted distance involves computing the pairwise distances between partition components, and subsequently solving a linear assignment problem to find the minimum cost matching. Below, we formulate the lifted distance using only one linear program.

\begin{proposition}
\label{prop:simplified_lp}
The lifted distance $A(X, Y)$ between partitions $X,Y$ satisfies
\begin{equation}\label{eq:combined}
A(X, Y) = \left\{
\begin{array}{rl}
\min\limits_{Q \in \R^{|E| \times k^2}, S \in \R^{k \times k}} &\sum\limits_{ij} \sum\limits_{e\in E} | Q_{ij}^e|\\
\st
&S\in\DS_k\\
& P^\top Q_{ij} - (x_i - y_j)S_{ij} = 0.
\end{array}
\right.
\end{equation}
\end{proposition}
\begin{proof} 
Substituting the transport cost \eqref{eq:w1graph} into \eqref{eqn:A definition 1}, we can write 
$$
A(X, Y) = 
\left\{
\begin{array}{rl}
\min\limits_{S\in \R^{k\times k}} &\sum\limits_{ij} S_{ij} \sum\limits_{e\in E} | J_{ij}^{*e}|\\
\st &S\in\DS_k\\
& J_{ij}^* = 
\left\{
\begin{array}{rl}
\argmin\limits_{J_{ij}\in \R^{|E|}} & \sum\limits_{e\in E} |J_{ij}^e|\\
\st &  P^\top J_{ij} = (x_i - y_j) .
\end{array}
\right.
\end{array}\right.
$$

Since the inner and outer problems are both minimizations, we can simplify this to
$$
A(X, Y) =
\left\{
\begin{array}{rl}
\min\limits_{S \in \R^{k\times k}, J\in \R^{|E|}} &\sum_{ij} S_{ij} \sum\limits_{e\in E} | J_{ij}^{e}|\\
\st &S\in\DS_k\\
& P^\top J_{ij} = (x_i - y_j) 
\end{array}
\right.
$$
This is a quadratic program in $S$ and $J$. Substituting $Q_{ij}^e = S_{ij}J_{ij}^e$ yields the desired result.

\end{proof}

The alternative formula in \cref{prop:simplified_lp} computes the distances between components and the linear assignment problem in a single linear program, suggesting an alternative means for computing $A(\cdot,\cdot)$ without first computing pairwise costs. Even so, $S$ gives the minimum matching of the components of $X$ and $Y$, and $Q$ represents the flow along each edge.

A standard linear programming duality argument applied to \eqref{eq:combined} shows 

\begin{equation}\label{eq:fulldual}
    A(X,Y) = 
    \left\{
    \begin{array}{rl}
    \max\limits_{\phi, \psi \in \R^k, \gamma \in \R^{|V| \times k^2}}   &\1^\top(\phi + \psi) \\
    \st & \phi_i + \psi_j \leq \gamma_{ij}^\top(x_i - y_j)\\
    & |\gamma_{ij}^{w} - \gamma{ij}^{v}| \leq 1 \ \forall \ (w, v) \in E.
    \end{array}
    \right.
\end{equation}

Similar to the argument in \cref{prop:simplified_lp}, this formula also can be derived directly by substituting the dual of \eqref{eq:w1graph} into the dual of \eqref{eqn:A definition 1}.

\section{The distance on partitions: unbalanced case}\label{section:unbalanced_distances}

In this section, we revisit the construction in \cref{section:distance mass balanced partitions} to propose a distance between graph partitions in a setting where the total mass of each component may not be equal.  We call these \textit{unbalanced partitions}. Similar to the balanced case, we define the distance in two steps: computing distances between partition components and using a linear assignment to lift the distance between components to a distance between partitions. In the general case, the distance between partition components is a modified transport distance that allows for mass to be inserted or removed at vertices with some cost. 
We show that lifting the unbalanced distance between components gives a valid metric on the space of unbalanced graph partitions and prove some basic properties of the general distance. 

\subsection{Distances between unbalanced components}\label{subsec:unbalanced_component_distances} Let $G = (V, E, \omega)$ be a weighted graph, and let $(V_1, \hdots, V_k)$ be a partition of the vertices of $G$. We define an unbalanced representation of $(V_1, \hdots, V_k)$ in $M(V)^{k*}$ as follows: To every component $V_i$, we associate a vector $x_i \in \R^{|V|}$ such that 
\begin{equation*} x_i(v) = \begin{cases} \omega(v) & \text{if} \ v \in V_i \\ 0 & \text{otherwise.} \end{cases} \end{equation*}
Then, $X = (x_1, \hdots, x_k) \in M(V)^{k*}$ gives an unbalanced representation of the partition $(V_1, \hdots, V_k)$. For example, in our target application of redistricting, the vertices of the graph correspond to geographic units such as census blocks and $\omega$ might represent populations associated with these units, which typically are balanced between voting districts but not identical from one unit to the next. 

We address this first at the level of $\M(V)$ with some inspiration from the formulations of unbalanced optimal transport in  \cite{chizat2018scaling,lombardi2015eulerian}. Let $x,y \in \M(v)$. For $p\geq 1$ and $\lambda>0$, the unbalanced problem consists of minimizing (see \cref{section:unbalanced transport p=1} for further discussion when $p=1$)
\begin{equation}\label{eqn:unbalanced problem for components}
    \begin{array}{rl}
    \min\limits_{J\in \R^{|E|}, z \in \R^{|V|}} & \|J\|_1 + \lambda \|z\|_p\\
    \st & P^\top J = y-x + z.
    \end{array}
\end{equation}
Then, we introduce a distance function on the space $M(V)$ of unbalanced partition components of a graph $G = (V, E, \omega)$ as follows. 

Let $X = (x_1, \hdots, x_k), Y = (y_1, \hdots, y_k) \in M(V)^{k*}$ be unbalanced partitions of $G$. The distance between components $x_i$ and $y_j$ is defined as
\begin{equation}\label{eqn:unbalanced_cost}
  C_{\lambda, p}(x_i,y_j)= \left\{
    \begin{array}{rl}
    \min\limits_{J\in \R^{|E|}, z \in \R^{|V|}} & \|J\|_1 + \lambda \|z\|_p\\
    \st & P^\top J = y_j-x_i + z\\
    \end{array}\right.
\end{equation}
where $\lambda \geq 0$ and $p \geq 1$ are parameters of the distance function $C_{\lambda, p}$. The variable $z$ allows slack in the amount of mass transported to or from each vertex under the transport plan. 
The parameters $\lambda$ and $p$ determine the weight of $z$ relative to $J$ in the objective. In \S \ref{subsec:unbalanced_basic_properties}, we discuss how the choice of $\lambda$ affects $C_{\lambda, p}$. 

In the following proposition, we show that $C_{\lambda, p}$ defines a valid metric on $M(V)$.:

\begin{proposition}
The function $C_{\lambda, p}(x,y)$ is a metric on $M(V)$ when $\lambda \geq 0$ and $p \geq 1$.
\end{proposition}
\begin{proof} It is immediate that $C_{\lambda, p}$ is nonnegative and symmetric. 

Let $x, y \in M(V)$ and suppose $x = y$. Then, $J = 0$ and $z = 0$ are feasible for $C_{\lambda, p}$, so $C_{\lambda, p} = 0$. Conversely, let $x, y \in M(V)$ and suppose $C_{\lambda, p} = 0$. Then, $J = 0$ and $z = 0$, so the constraint $P^\top J = y - x + z$ implies that $x = y$. 

Now, suppose $x, y, w \in M(V)$ and let
\begin{align*}
C_{\lambda, p}(x,y) &= \|J_{xy}\|_1 + \lambda\|z_{xy}\|_p \\
C_{\lambda, p}(x,w) &= \|J_{xw}\|_1 + \lambda\|z_{xw}\|_p \\
C_{\lambda, p}(w,y) &= \|J_{wy}\|_1 + \lambda\|z_{wy}\|_p
\end{align*}
Then, 
\begin{equation*}
P^\top(J_{xw} + J_{wy}) = y - x + z_{xw} + z_{wy},
\end{equation*}
so $J_{xw} + J_{wy}$ and $v_{xw} + v_{wy}$ are feasible for $C_{\lambda, p}(x,y)$. Since $J_{xy}$ and $v_{xy}$ are optimizers for $C_{\lambda, p}(x, y)$, we have
\begin{align*}
\|J_{xy}\|_1 + \lambda\|v_{xy}\|_p &\leq \|J_{xz} + J_{zy}\|_1 + \lambda\|v_{xz} + v_{zy}\|_p \\
&\leq \|J_{xz}\|_1 + \|J_{zy}\|_1 + \lambda\|v_{xz}\|_p + \lambda\|v_{zy}\|_p.
\end{align*}
Therefore, 
\begin{align*}
C_{\lambda, p}(x,y) \leq C_{\lambda, p}(x,z) + C_{\lambda, p}(z,y),
\end{align*}
so $C_{\lambda, p}$ satisfies the triangle inequality. 
\end{proof}

\subsection{Distances between unbalanced partitions}\label{subsec:unbalanced_parition_distances}

In the previous section, we defined a distance on the space of unbalanced partition components. In this section, we extend the lifted distance construction from \cref{subsec:balanced_part_distances} to a general lifted distance between potentially unbalanced partitions. 

\begin{definition}[Unbalanced lifted distance]
The \emph{unbalanced lifted distance} $$A_{\lambda, p} : M(V)^{k*} \times M(V)^{k*} \to \R$$ between partitions $X$ and $Y$ is defined as 
\begin{equation} \label{eqn:unbalanced_lp} A_{\lambda, p}(X, Y) = \left\{
\begin{array}{rl}
  \min\limits_{S\in \R^{k \times k}} & \sum\limits_{ij} S_{ij} C_{\lambda, p}(x_i, y_j) \\
                 \st & S\in\DS_k,
\end{array}\right.
\end{equation}
where $X = (x_1, \hdots, x_k), Y = (y_1, \hdots, y_k) \in M(V)^{k*}$ and $C_{\lambda, p}$ is the unbalanced distance between partitions in \eqref{eqn:unbalanced_cost}. 
\end{definition}

Just as in the balanced case, the unbalanced lifted distance uses a linear assignment problem to find a minimum-cost matching of the partition components. The unbalanced lifted distance induces a valid metric on the space of unbalanced graph partitions. 

\begin{proposition}\label{prop:unbalanced_metric} The unbalanced lifted distance $A_{\lambda, p}$ defined in \eqref{eqn:unbalanced_lp} is a metric on $M(V)^{k*}$. \end{proposition} 
The proof of \cref{prop:unbalanced_metric} follows from the proof of \cref{prop:plan_metric}.

\subsection{Basic properties}\label{subsec:unbalanced_basic_properties}

In this section, we prove several properties of the unbalanced lifted distance $\eqref{eqn:unbalanced_lp}$. First, we show that both the distance $C_{\lambda, p}$ and the lifted distance $A_{\lambda, p}$ are monotonic in $\lambda$:

\begin{proposition} 
\label{prop:cost_monotonic}
The unbalanced cost function $C_{\lambda, p}$ is monotonic in $\lambda$.
\end{proposition}
\begin{proof}

Suppose $\lambda_1, \lambda_2 \geq 0$ and $\lambda_2 > \lambda_1$. Let $J_{2}$ and $z_{2}$ be optimizers for $C_{\lambda_2, p}(x_i, y_j)$.
Then, $J_{2}$ and $z_{2}$ are feasible for $C_{\lambda_1, p}(x_i, y_j)$, so 
\begin{align*} 
C_{\lambda_1, p}(x_i, y_j) 
\leq \sum_{e} |J_{2}^e| + \lambda_1 \|z_{2}\|_p 
\leq \sum_{e} |J_{2}^e| + \lambda_2 \|z_2\|_p 
\leq C_{\lambda_2, p}(x_i, y_j). \end{align*}
Therefore, if $\lambda_2 > \lambda_1$, we have $C_{\lambda_2, p} \geq C_{\lambda_1, p}$. 
\end{proof}
\begin{corollary}
The unbalanced lifted distance $A_{\lambda, p}$ is monotonic in $\lambda$.
\end{corollary}
\begin{proof}
Suppose $\lambda_1, \lambda_2 \geq 0$ and $\lambda_2 > \lambda_1$. Let $S_2$ be an optimizer for $A_{\lambda_2, p}$. Since $S_2$ is feasible for $A_{\lambda_1, p}$, 
\begin{equation*} A_{\lambda_1, p} \leq \sum_{ij} S_{2ij} C_{\lambda_1, p}(x_i, y_j) \leq \sum_{ij} S_{2ij}C_{\lambda_2, p}(x_i, y_j) = A_{\lambda_2, p}, \end{equation*}
as desired.
\end{proof}
In the following proposition, we show that the norm of the optimizer of the mass difference $z$ for $C_{\lambda, p}$ is monotonic in $\lambda$:
\begin{proposition} 
\label{prop:inf_node_monotonic}
The $p$-norm of the optimizer $z_{\lambda}$ for the unbalanced cost function $C_{\lambda, p}$ is monotonic in $\lambda$. 
\end{proposition}
\begin{proof}
Suppose $\lambda_1, \lambda_2 \geq 0$ and $\lambda_2 > \lambda_1$. Let $J_{1}$ and $z_{1}$ be optimizers for $C_{\lambda_1, p}$, and let $J_{2}$ and $z_{2}$ be optimizers for $C_{\lambda_2, p}$. Since $J_{2}$ and $z_{2}$ are feasible for $C_{\lambda_1, p}$,
\begin{equation*}
\|J_{1}\|_1 + \lambda_1\|z_{1}\|_p \leq \|J_{2}\|_1 + \lambda_1\|z_{2}\|_p \end{equation*}
and thus 
\begin{equation}\label{eqn:lambda_upperbound} \|J_{1}\|_1 - \|J_{2}\|_1 \leq \lambda_1(\|z_{2}\|_p - \|z_{1}\|_p). \end{equation}

By an identical argument,
\begin{equation}\label{eqn:lambda_lowerbound}
\|J_{1}\|_1 - \|J_{2}\|_1 \geq \lambda_2(\|z_{2}\|_p - \|z_{1}\|_p).
\end{equation}
Combining \eqref{eqn:lambda_upperbound} and \eqref{eqn:lambda_lowerbound} gives
$
\lambda_2(\|z_{2}\|_p - \|z_{1}\|_p) \leq \lambda_1(\|z_{2}\|_p - \|z_{1}\|_p).
$ 
Since $\lambda_2 > \lambda_1$ by assumption, we must have $\|z_{2}\|_p - \|z_{1}\|_p \leq 0$, as needed. 
\end{proof}

In \cref{prop:simplified_lp}, we give a formulation of the balanced lifted distance as a combined linear program. In the following proposition, we prove a combined convex program for the unbalanced lifted distance:

\begin{proposition} \label{prop:unbalanced_simplified_lp}
The unbalanced lifted distance $A_{\lambda, p}(X, Y)$ between partitions $X,Y$ satisfies
\begin{equation}\label{eqn:unbalanced_combined}
A_{\lambda, p}(X, Y) = \left\{
\begin{array}{rl}
\min\limits_{\substack{Q \in \R^{|E| \times k^2} \\ S \in \R^{k \times k} \\ u \in \R^{|V| \times k^2}}} &\sum\limits_{ij} \left(\sum\limits_{e\in E} | Q_{ij}^e| + \lambda \|u_{ij}\|_p\right)\\
\st
&S\in\DS_k\\
& P^\top Q_{ij} - (x_i - y_j)S_{ij} = u_{ij}.
\end{array}
\right.
\end{equation}
\end{proposition}

\begin{proof} This proof proceeds identically to the proof of \cref{prop:simplified_lp}.
Plugging the unbalanced transport cost \eqref{eqn:unbalanced_cost} into the unbalanced lifted distance \eqref{eqn:unbalanced_lp}, we get 
$$
A_{\lambda, p}(X, Y) = 
\left\{
\begin{array}{rl}
\min\limits_{S\in \R^{k\times k}} &\sum\limits_{ij} S_{ij} \left(\sum\limits_{e\in E} | J_{ij}^{*e}| + \lambda \|z_{ij}^*\|_p\right)\\
\st &S\in\DS_k\\
& J_{ij}^*, z_{ij}^* = 
\left\{
\begin{array}{rl}
\argmin\limits_{J_{ij}, z_{ij}} & \sum\limits_{e\in E} |J_{ij}^e| + \lambda \|z_{ij}\|_p\\
\st &  P^\top J_{ij} = x_i - y_j + z_{ij}.
\end{array}
\right.
\end{array}\right.
$$
Since the inner and outer problems are both minimizations, this simplifies to 
$$
A_{\lambda, p}(X, Y) =
\left\{
\begin{array}{rl}
\min\limits_{\substack{S\in \R^{k\times k}\\ J\in \R^{|E|\times k^2}\\ z\in \R^{|V|\times k}}} &\sum\limits_{ij} S_{ij} \left(\sum\limits_{e\in E} | J_{ij}^{e}| + \lambda \|z_{ij}\|_p\right)\\
\st &S\in\DS_k\\
& P^\top J_{ij} = x_i - y_j + z_{ij}
\end{array}
\right.
$$
The variable substitutions $Q_{ij}^e = S_{ij}J_{ij}^e$ and $u_{ij} = S_{ij}z_{ij}$ give the desired result. 
\end{proof}

The dual of \eqref{eqn:unbalanced_combined} is given by 
\begin{equation} 
    A_{\lambda, p}(X,Y) = 
    \left\{
    \begin{array}{rl}
    \max\limits_{\phi, \psi \in \R^k, \gamma \in \R^{|V| \times k^2}}   &\1^\top(\phi + \psi) \\
    \st & \phi_i + \psi_j \leq \gamma_{ij}^\top(x_i - y_j)\\
    & |\gamma_{ij}^{w} - \gamma_{ij}^{v}| \leq 1 \ \forall \ (w, v) \in E \\
    & \|\gamma_{ij}\|_q \leq \lambda
    \end{array}
    \right.
\end{equation}
where the $q$-norm, satisfying $q = \begin{cases} \frac{p}{p-1} & p > 1 \\ \infty & p = 1 \end{cases}$ is the dual of the $p$-norm in the objective of the primal problem. Slater's condition for duality \cite{slater1950lagrange} gives that strong duality holds for \eqref{eqn:unbalanced_combined}.  We can verify that Slater's condition holds for \eqref{eqn:unbalanced_combined} since we can write the inner problem \eqref{eqn:unbalanced_cost} as

\begin{equation}
  C_{\lambda, p}(x_i,y_j)= \left\{
    \begin{array}{rl}
    \min\limits_{\substack{J\in \R^{|E|\times k^2} \\ z \in \R^{|V|} \\ m\in \R}} & \|J\|_1 + \lambda m\\
    \st & P^\top J = y_j-x_i + z\\
        & m\geq ||z||_p.
    \end{array}\right.
\end{equation}

The final constraint is the only nonlinear one, and for any solution satisfying the linear constraints, we can choose $m$ to be large enough that this solution satisfies the nonlinear constraint with strict inequality.  Using this as the inner problem to the linear program which computes the optimal matching does not introduce any additional nonlinear constraints, so Slater's condition is satisfied for \eqref{eqn:unbalanced_combined} and strong duality holds.

Next, we show that for balanced partitions, the general distance is an extension of the balanced distance. Specifically, if we take $p = 1$, there exists $\lambda$ sufficiently large such that the unbalanced distance between two balanced partition components is equal to the balanced distance. The following proposition formalizes this notion:

\begin{proposition}
\label{prop:large_lambda}
If $X$ and $Y$ are balanced partitions of a graph $G$ and $\lambda \geq \diam(G)/2$, then $C_{\lambda, 1}(x_i, y_j) = W_1(x_i, y_j)$ for $x_i \in X$, $y_j \in Y$. 
\end{proposition}
\begin{proof}

Suppose that $\bar{J}$ and $\bar{z}$ are optimizers for $C_{\lambda, 1}(x_i, y_j)$. We know that 
\begin{equation*} \sum_{v \in V} x_i(v) + \sum_{v \in V} \bar{z}(v) = \sum_{v \in V} y_j(v), \end{equation*} 
so when $X$ and $Y$ are balanced, $\sum_{v \in V} \bar{z}(v) = 0$. 
Let $\bar{z}^+$ and $\bar{z}^-$ be defined such that 
\begin{equation*} \bar{z}^+(v) = \begin{cases} z(v) & z(v) > 0 \\ 0 & \text{otherwise} \end{cases} \ \ \text{and} \ \ \ \bar{z}^-(v) = \begin{cases} |z(v)| & z(v) < 0 \\ 0 & \text{otherwise} \end{cases}. \end{equation*}
Let $J^*$ represent the optimal transport plan between $\bar{z}^-$ and $\bar{z}^+$. Then,
\begin{equation*} \|J^*\|_1 \leq \diam(G) \sum_{v \in V} |\bar{z}^-(v)| = \frac{\diam(G)}{2} \sum_{v \in V} |\bar{z}(v)|.\end{equation*}
Therefore, 
\begin{equation}\label{eqn:upper} \|\bar{J}\|_1 + \|J^*\|_1 \leq \|\bar{J}\|_1 + \frac{\diam(G)}{2} \|\bar{z}\|_1 \leq C_{\lambda, 1}(x_i, y_j). \end{equation} 
Since \begin{equation*}P^T(\bar{J} + J^*) = (x_i - y_j + \bar{z}) + (\bar{z}^- - \bar{z}^+) = x_i - y_j,\end{equation*} $\bar{J} + J^*$ is feasible for $W_1(x_i, y_j)$, so 
\begin{equation}\label{eqn:lower} W_1(x_i, y_j) \leq \|\bar{J} + J^*\|_1 \leq \|\bar{J}\|_1 + \|J^*\|_1. \end{equation}
Combining \eqref{eqn:upper} and \eqref{eqn:lower}, we get $W_1(x_i, y_j) \leq C_{\lambda, 1}(x_i, y_j)$. It is immediate that in the balanced case, $C_{\lambda, 1}(x_i, y_j) \leq W_1(x_i, y_j)$, so $W_1(x_i, y_j) = C_{\lambda, 1}(x_i, y_j)$, as desired. 
\end{proof}
We can extend the previous result to show that for $\lambda \geq \diam(G)/2$, the unbalanced lifted distance is equal to the balanced lifted distance for balanced partitions. 
\begin{corollary}
If $X$ and $Y$ are balanced partitions of a graph $G$ and $\lambda \geq \diam(G)/2$, then $A_{\lambda, 1}(X, Y) = A(X, Y)$. 
\end{corollary}
\begin{proof}
When $W_1(x_i, y_j) = C_{\lambda, 1}(x_i, y_j)$ for all $x_i \in X$, $y_j \in Y$, the definition of lifted distance given in \eqref{eqn:A definition 1} is the same as the definition of unbalanced lifted distance given in \eqref{eqn:unbalanced_lp}. 
\end{proof}

\subsection{The unbalanced transport problem when $p=1$}\label{section:unbalanced transport p=1}

In this case the unbalanced transport problem \eqref{eqn:unbalanced problem for components} is equivalent to a discrete version of the transportation problem with boundary studied by Figalli and Gigli \cite{figalli2010transportation}, which is to \eqref{eqn:unbalanced problem for components} what the Kantorovich problem in \eqref{eqn:transport metric definition} is to \eqref{eq:w1graph}. The idea is to modify the Kantorovich problem by adding vertices that serve as auxiliary infinite-capacity sinks/reservoirs that can receive or provide mass to compensate for unequal total masses between $x,y \in \M(V)$.

In our case, we expand the graph $G = (V,E,\omega)$ by adding just one extra vertex denoted $v_s$ to serve as the auxiliary vertex. We also add edges between every vertex and $v_s$, all with the the same weight $\lambda>0$. Concretely, we define $G_* = (V_*,E_*,\omega_*)$ as follows:
\begin{align*}
  V_* &:= V \cup \{ v_s\},\\
  E_* &:= E \cup V \times \{v_s\} \cup \{v_s\} \times V,\textrm{ and}\\
  \omega_*(e) &:= \left \{ \begin{array}{rl}
    \omega(e) & \text{ if } e \in E,\\
    \lambda & \text{ if } e \in E_* \setminus E.
  \end{array}\right.
\end{align*}
We denote by $d_{\lambda}(v,w)$ the resulting graph distance in $V^*$ (note that this simply extends the distance in $G$ via $d(v,v_s) = \lambda$ for every $v\in V$). 

Given $x,y \in \M(V)$, we will say that $\pi \in \M(V\times V)$ is an admissible transport plan for $x$ and $y$ with sink at $v_s$ if 
\begin{align*}
  \sum \limits_{w \in V_*} \pi(v,w) = x(v)\;\;\forall\;v \in V\;\;
  \textrm{and}
  \;\;\sum \limits_{v\in V_*} \pi(v,w) = y(w)\;\;\forall\;w\in V.
\end{align*}
This condition is similar to the usual Kantorovich problem from \eqref{eqn:transport metric definition}, except in our larger space we do not impose the marginal constraint at the auxiliary vertex $v_s$; this means we are free to move any amount of mass to, from, or through $v_s$ as needed. 
The set of such admissible plans will be denoted $\Pi_*(x,y)$. 

Then, the analogue of the Kantorovich problem is
\begin{align}\label{eqn:OPT boundary}
  \min \limits_{\pi \in \Pi_*(x,y)} \sum\limits_{v,w \in V_*} d_{\lambda}(v,w)\pi(v,w).
\end{align}
We will show that Problem \eqref{eqn:OPT boundary} is equivalent to Problem \eqref{eqn:unbalanced problem for components} when $p=1$. 
\begin{lemma}\label{lem:OPT boundary and OPT unbalanced are the same} 
  The minimum for Problem \eqref{eqn:OPT boundary} is the same as the minimum for the problem
  \begin{align}\label{eqn:Beckmann unbalanced}
    \min_{J \in \R^{|E|}}\hspace{5mm} &\|J\|_1 + \lambda \|z\|_1 \\ \st\hspace{5mm} &P^\top J = y - x + z. \notag
  \end{align}
  Moreover, from any $\pi$ which is a minimizer of Problem \eqref{eqn:OPT boundary} it is possible to construct a corresponding pair $(J,z)$ which is a minimizer for Problem \eqref{eqn:Beckmann unbalanced}.
  
\end{lemma}

To prove \cref{lem:OPT boundary and OPT unbalanced are the same}, let us make some preliminary observations. The essence of the proof lies in the following construction, which is commonly used to prove the equivalence between \eqref{eqn:transport metric definition} and \eqref{eq:w1graph} (see \cite[\S4.2]{santambrogio2015optimal}). For every pair $v,w\in V_*$, choose a minimal path from $v$ to $w$ and denote by $E(v,w) \subset E_*$ the set of edges that appear in this path. That is, if the minimal path chosen for $v$ and $w$ is given by $v = v_0,\ldots,v_N = w$, then $E_*(v,w) = \{ (v_0,v_1),(v_1,v_2),\ldots,(v_{N-1},v_N)\}$. Then, given any $\pi \in M(V_*\times V_*)$ we define $J_{\pi}:E_* \to\mathbb{R}$ and $z_\pi:V_* \to\mathbb{R}$ as follows

\begin{align}\label{eqn:J_pi and z_pi}
  J_{\pi,e} := \sum\limits_{v\in V_*}\sum \limits_{w\in V_*} \mathbbm{1}_{E(v,w)}(e)\pi(v,w),\;\; z_\pi(v) := \pi(v,v_s) - \pi(v_s,v).
\end{align}

The proof of \cref{lem:OPT boundary and OPT unbalanced are the same} boils down to showing that if $\pi$ is a minimizer for \eqref{eqn:OPT boundary} then $(J_\pi,z_\pi)$ given by \eqref{eqn:J_pi and z_pi} is a minimizer for \eqref{eqn:Beckmann unbalanced}. We start by showing $(J_\pi,z_\pi)$ is an admissible pair.
\begin{proposition}\label{proposition:J_pi is admissible}
  Let $x,y \in \M(V)$. If $\pi \in \Pi_*(x,y)$, then 
  \begin{align*}
    (P^{\top}J_{\pi})(v) & = y(v)-x(v) + z_\pi(v) \textnormal{ for } v \in V.
  \end{align*}
  Moreover, we have, with $P_*$ denoting the incidence matrix for the graph $G_*$,
  \begin{align*}
    (P_*^{\top}J_{\pi})(v) & = y(v)-x(v) \textnormal{ for } v \in V.
  \end{align*}
\end{proposition}

\begin{proof}
  Let $E_{\textnormal{in}}(v)$ and $E_{\textnormal{out}(v)}$ denote the sets of the incoming and outgoing edges of vertex $v$, respectively, then 
  \begin{align*}
    (P^{\top}J_{\pi})(v) = \sum \limits_{e \in E_{\textnormal{in}(v)}}J_{\pi,e} - \sum \limits_{e \in E_{\textnormal{out}(v)}}J_{\pi,e}.
  \end{align*}
  Fix $v_0,w_0 \in V$, and let $\pi_0$ be the function
  \begin{align*}
    \pi_0(v,w) =
    \left\{\begin{array}{ll}
    1 & \textnormal{ if } (v,w) = (v_0,w_0)\\ 0& \textnormal{ otherwise}.
    \end{array}
    \right.
  \end{align*}
  Then, we have 
  \begin{align*}
     \sum \limits_{e \in E_{\textnormal{out}(v_0)}}J_{\pi,e} - \sum \limits_{e \in E_{\textnormal{in}(v_0)}}J_{\pi,e} &= 1,\\
     \sum \limits_{e \in E_{\textnormal{out}(w_0)}}J_{\pi,e} - \sum \limits_{e \in E_{\textnormal{in}(w_0)}}J_{\pi,e} &= -1,\textrm{ and}\\
     \sum \limits_{e \in E_{\textnormal{out}(v)}}J_{\pi,e} - \sum \limits_{e \in E_{\textnormal{in}(v)}}J_{\pi,e} &= 0 \textnormal{ if } v \neq v_0,w_0. 
  \end{align*}
  From a linear combination of these identities for each pair $(v_0,w_0) \in V\times V$ we obtain the following formula for any $\pi \in M(V\times V)$:
  \begin{align*}
    \sum \limits_{e \in E_{\textnormal{out}(v)}}J_{\pi,e} - \sum \limits_{e \in E_{\textnormal{in}(v)}}J_{\pi,e} = \sum \limits_{w \in V_*} \pi(w,v) - \sum \limits_{v \in V_*} \pi(v,w).    
  \end{align*}
  Now, if $\pi$ is an admissible plan, we have
  \begin{align*}
    \sum \limits_{e \in E_{\textnormal{out}(v)}}J_{\pi,e} - \sum \limits_{e \in E_{\textnormal{in}(v)}}J_{\pi,e} & = \sum \limits_{w \in V} \pi(w,v) - \sum \limits_{w \in V} \pi(v,w) + \pi(v_s,v) - \pi(v,v_s)\\
    & = x(v) - y(v) - z_\pi(v).
  \end{align*}
  It follows that 
  \begin{align*}
    (P^{\top}J_{\pi})_v = y(v)-x(v) + z_\pi(v).
  \end{align*}
  This proves the first identity. For the second one, observe that 
  \begin{align*}
    (P_*^{\top}J_{\pi})_v & = \sum \limits_{e\in E_*} P_{ev} J_{\pi,v} \\
	  & = \sum \limits_{e\in E} P_{ev} J_{\pi,v}  + \sum \limits_{w \in V} P_{(w,v_s)v}J_{\pi,w} + \sum \limits_{w \in V} P_{(v_s,w)v}J_{\pi,w}\\
	  & = (P^{\top}J_{\pi})(v) - J_{\pi,(v,v_s)} + J_{\pi,(v_s,v)}.
  \end{align*}
  Using that $J_{\pi,(v,v_s)} = \pi(v,v_s)$ and $J_{\pi,(v_s,v)} = \pi(v_s,v)$ together with the formula for $(P^{\top}J_{\pi})_v$, we obtain
  \begin{align*}
    (P_*^{\top}J_{\pi})_v & = y(v)-x(v) + z_\pi(v) - \pi(v,v_s) + \pi(v_s,v)\\
	  & = y(v)-x(v),
  \end{align*}
  and the second formula is proved.
\end{proof}

With this, we are ready to prove the equivalence between the two problems.
\begin{proof}[Proof of  \cref{lem:OPT boundary and OPT unbalanced are the same}]
  Let $\pi$ be a minimizer for Problem \eqref{eqn:OPT boundary}. According to \cref{proposition:J_pi is admissible}, $(J_{\pi},z_{\pi})$ is an admissible pair for Problem \eqref{eqn:Beckmann unbalanced}. Therefore,
  \begin{align*}  
    \|J_\pi\|_1 + \lambda \|z_\pi\|_1 \geq 
    \left\{
    \begin{array}{ll}
    \min_{J \in \R^{|E|}} &\|J\|_1 + \lambda \|z\|_1 \\ \st &P^\top J = y - x + z.
    \end{array}
    \right.
  \end{align*}
  We have $J_{\pi,e} \geq 0$ for every $e$, and hence
  \begin{align*}
    \sum \limits_{e\in E_*}|J_{\pi,e}|\omega(e) = \sum \limits_{e\in E_*}J_{\pi,e}\omega(e) & = \sum \limits_{e \in E_*}\sum\limits_{v\in V_*}\sum \limits_{w\in V_*} \mathbbm{1}_{E(v,w)}(e) \omega(e)\pi(v,w)\\
	& = \sum\limits_{v\in V_*}\sum \limits_{w\in V_*} \left ( \sum \limits_{e \in E} \mathbbm{1}_{E(v,w)}(e) \omega(e)\right )\pi(v,w).
  \end{align*}
  From the definition of the sets $E(v,w)$,  for any $v,w\in V_\infty$ we have
  \begin{align*}
    \sum \limits_{e\in E_*}\mathbbm{1}_{E(v,w)}(e) \omega(e) = d_{\lambda}(v,w).
  \end{align*}
  Therefore
  \begin{align*}
    \sum \limits_{e\in E_*}|J_{\pi,e}|\omega(e)  = \sum\limits_{v\in V_*}\sum \limits_{w\in V_*}d_{\lambda}(v,w)\pi(v,w).
  \end{align*}
  On the other hand, the sum on the left can be decomposed as
  \begin{align*}
    \sum \limits_{e\in E_*}J_{\pi,e}\omega(e) = \sum \limits_{e\in E}J_{\pi,e}\omega(e) + \sum \limits_{v\in V} J_{\pi,(v,v_s)}\omega(v,v_s)+ \sum \limits_{v\in V} J_{\pi,(v_s,v)}\omega(v_s,v).
  \end{align*}
  Since $\omega(v_s,v) = \omega(v,v_s) = \lambda$ for every $v\in V$,
  \begin{align*}
    \sum \limits_{e\in E_*}J_{\pi,e}\omega(e) = \sum \limits_{e\in E}J_{\pi,e}\omega(e) + \lambda \sum \limits_{v\in V} J_{\pi,(v,v_s)} +  J_{\pi,(v_s,v)}.
  \end{align*}
  From the definition of $J_{\pi,e}$, we have 
  \begin{align*}
    J_{\pi,(v_s,v)} = \pi(v_s,v) \textnormal{ and } J_{\pi,(v,v_s)} = \pi(v,v_s) \textnormal{ for every } v \in V.
  \end{align*}
  The minimizer $\pi$ can always be modified so that for every $v$ at most one of $\pi(v,v_s)$ and $\pi(v_s,v)$ is non-zero. In this case $|z_\pi(v)| = \pi(v,v_s)+\pi(v_s,v)$ for every $v\in V$, and
  \begin{align*}
    \sum \limits_{e\in E_*}J_{\pi,e}\omega(e) = \sum \limits_{e\in E}J_{\pi,e}\omega(e) + \lambda \sum \limits_{v\in V} |z_\pi(v)|.
  \end{align*}
  This shows that 
  \begin{align*}  
    \sum \limits_{v\in V_*}\sum \limits_{w\in V_*}d_\lambda(v,w) \pi(v,w) =  \|J_\pi\|_1 + \lambda \|z_\pi\|_1,
  \end{align*}
  which shows the minimum for Problem \eqref{eqn:OPT boundary} is no smaller than the minimum for Problem \eqref{eqn:Beckmann unbalanced}. 
  
  For the reverse inequality we will implicitly use the dual problem to \eqref{eqn:OPT boundary}. Consider pairs of functions $\phi,\psi:V_* \to \mathbb{R}$ such that $\phi(v_s) = \psi(v_s) =0$ and for every $v,w \in V_*$
  \begin{align}\label{eqn:OPT boundary 1-Lipschitz}
    \phi(v)+\psi(w) \leq d_\lambda(v,w).
  \end{align}
  Following  \cite[Appendix A]{guillen2019coupling}, the dual problem to \eqref{eqn:OPT boundary} is maximizing the functional
  \begin{align*}
    \sum \limits_{v\in V}\phi(v)x(v) + \sum \limits_{w\in V}\psi(w)y(w)
  \end{align*}
  over all pairs $\phi,\psi$ described above. Let us show that the minimum of \eqref{eqn:Beckmann unbalanced} is larger than this for any $\phi,\psi$. Without loss of generality, we may assume that $\phi$ is such that
  \begin{align*}
    \phi(v) = \min \limits_{w\in V^*} d_\lambda(v,w) -\psi(w).      
  \end{align*}
  In this case it is easy to see that $|\phi(v)-\phi(v)| \leq d_{\lambda}(v,w)$ for every $v$ and $w$. In particular, if $e = (v,w)$ is an edge we have $|\phi(w)-\phi(v)| \leq d_{\lambda}(v,w) = \omega(e)$, since $(P_*\phi)_e = \phi(w)-\phi(v)$ this shows that
  \begin{align*}
     |(P_* \phi)_e| \leq \omega(e) \textnormal{ for every } e\in E_*.
  \end{align*}
  Combining these inequalities for each $e\in E_*$ and using the dual of $P_*$, we have
  \begin{align*}
    \sum \limits_{e\in E_* }\omega(e)|J_e| & \geq - \sum \limits_{e\in E_*}J_e (P_* \phi)_e = \sum \limits_{v\in V_*} (P_*^{\top}J)_v \phi(v).
  \end{align*}
  Since $\phi(v_s) = 0$ and $(P_*^{\top}J)_v = y(v)-x(v)$ when $v\neq v_s$, it follows that 
  \begin{align*}
	  \sum \limits_{e\in E_* }\omega(e)|J_e| & \geq \sum \limits_{v\in V} \phi(v) (x(v)-y(v)) \\
	    & = \sum \limits_{v\in V}\phi(v)x(v) - \sum \limits_{v\in V}\phi(v)y(v).
  \end{align*}
  On the other hand, applying \eqref{eqn:OPT boundary 1-Lipschitz} with $v=w$ yields the inequality $\psi(v) \leq -\phi(v)$ for every $v$, from where it follows that 
  \begin{align*} 
	  \sum \limits_{v\in V} \phi(v)x(v) - \sum \limits_{v \in V}\phi(v)y(v)  \geq \sum \limits_{v\in V} \phi(v)x(v) + \sum \limits_{v \in V}\psi(v)y(v).	 
  \end{align*}
  In conclusion, for every admissible pair $\phi$ and $\psi$ we have the inequality 
  \begin{align*}
    \sum \limits_{e\in E_*}|J_e|\omega(e) & \geq \sum \limits_{v\in V} \phi(v)x(v) + \sum \limits_{v \in V}\psi(v)y(v).
  \end{align*}  
  Taking the supremum over all admissible $\phi$ and $\psi$ we have, by duality, 
  \begin{align*}
    \sum \limits_{e\in E_*}|J_e|\omega(e) & \geq \inf \limits_{\pi \in \Pi_*} \sum\limits_{v\in V_*}\sum\limits_{w\in V_*}d_\lambda(v,w)\pi(v,w),
  \end{align*}  
  and this finishes the proof.

\end{proof}

\section{Bounds}\label{section:bounds}

We can relate our distance on partitions to other constructions in the literature: the \textit{Hamming distance} and the \textit{total variation distance}.

In information theory, the Hamming distance between two binary strings of equal length is the number of positions in which they differ. Inspired by this definition, the authors in \cite{Mattingly2018} compute a notion of Hamming distance between two graph partitions $X$ and $Y$.  Using $v\in x_i$ to indicate that vertex $v$ belongs to component $i$ of partition $X$, the Hamming distance is defined as
\begin{equation}
    \mathrm{dist}_\mathrm{HAM} (X,Y) = \left\{
    \begin{array}{rl}
    \min\limits_{S\in \R^{k\times k}}& \sum\limits_{v\in V} \sum\limits_{ij} S_{ij} \mathbbm{1}\left[ v\in x_i \land v\notin y_j     \right]\\
    \st & S\in \DS_k.
    \end{array}
    \right.
\end{equation}

That is, for a given matching of the components, we count the number of vertices whose label differs and take the minimum over all matchings. 
Generalizing to non-binary functions for the weights on vertices, we can formulate a distance as the sum of the vertexwise differences in weights over the matched components.  First, we can write the $L_1$ or \textit{total variation} distance between two components $x_i$ and $y_j$ as 

\begin{equation}
    \ell_{1}(x_i,y_j) = \frac{1}{2}\sum\limits_{v\in V} | x_i(v) - y_j(v)  |
\end{equation}
and we can lift this to a distance between partitions by solving the assignment problem using $\ell_1$ as the cost function.  We write

\begin{equation}
\label{eqn:l1_lift}
    {L_{1}} (X,Y) = \left\{
    \begin{array}{rl}
    \min\limits_{S\in \R^{k\times k}} &\sum\limits_{v\in V} \sum\limits_{ij}S_{ij} \ell_1(x_i,y_j)\\
    \st& S\in \DS_k,
    \end{array}
    \right.
\end{equation}
These distances are of the same form as \eqref{eqn:A definition 1} but with different cost functions $C$.

We can show that, for any choice of weight on the vertices, the $L_1$ distance lower-bounds the transport distance between partitions:
\begin{proposition}
For two balanced partitions $X$ and $Y$, we have ${L_1}(X,Y) \leq A(X,Y).$
\end{proposition}

\begin{proof}
Consider a pair of components $x_i$ and $y_j$. From \eqref{eqn:transport metric definition}, we know that 
\begin{equation*} W_1(x_i, y_j) = \min_{\pi \in \Pi(x_i, y_j)} \sum_{v, w \in V} d(v, w) \pi(v, w). \end{equation*}
Because $d(v, w)$ is the shortest path distance between $v$ and $w$, $d(v, w) = 0$ if and only if $v = w$, and $d(v, w) \geq 1$ otherwise. Therefore, for any $\pi \in \Pi(x_i, y_j)$, 
\begin{equation*}  \sum_{v, w \in V} d(v, w) \pi(v, w) =  \sum_{v \neq w} d(v, w) \pi(v, w) \geq \sum_{v \neq w} \pi(v, w). \end{equation*}
Let $\pi^*$ be the element of $\Pi(x_i, y_j)$ that minimizes $W_1(x_i, y_j)$. Then,

\begin{align*} \sum_{v \neq w} \pi^*(v, w) &= \sum_{v, w \in V} \pi^*(v, w) - \sum_{v \in V} \pi^*(v, v)\\
&= \sum_{v \in V} \left( x_i(v) - \pi^*(v, v) \right). \end{align*}

Because $\pi^*$ moves as little mass as possible, $\pi^*(v, v) = \min(x_i(v), y_j(v))$. Then, 
\begin{equation*} x_i(v) - \pi^*(v, v) = \begin{cases} 0 &\textrm{ if } x_i(v) \leq y_j(v) \\ x_i(v) - y_j(v) & \textrm{ otherwise,}
\end{cases} \end{equation*}
and hence
\begin{equation*} \sum_{v \neq w} \pi^*(v, w) = \sum_{\substack{v \in V \\ x_i(v) \geq y_j(v)}} x_i(v) - y_j(v) = \frac{1}{2} \sum_{v \in V} |x_i(v) - y_j(v)|. \end{equation*}
Therefore, we have shown that 
\begin{equation*} \min_{\pi \in \Pi(x_i, y_j)} \sum_{v, w \in V} d(v, w) \pi(v, w) \geq \frac{1}{2} \sum_{v \in V} |x_i(v) - y_j(v)|, \end{equation*} so $W_1(x_i, y_j) \geq \ell_1(x_i, y_j)$. 

Now, let $S$ denote the optimal matching for the formulation of $A(X, Y)$ given in $\eqref{eqn:A definition 1}$. Then, 
\begin{equation*} A(X, Y) = \sum_{ij} S_{ij}W_1(x_i, y_j) \geq \sum_{ij}S_{ij} \ell_1(x_i, y_j) \geq L_1(X, Y), \end{equation*} 
as desired. \end{proof}

Qualitatively, we expect some differences between the $L_1$ and transport distances.  First, the $L_1$ distance does not see the structure of the graph, since it is computed from only the overlapping portions of matched components in the two partitions.  For this reason, two qualitatively similar components with little overlap are as far apart in $L_1$ distance as two components on opposite sides of the graph, whereas the transport distance recognizes that the former are closer together than the latter. This can happen in practice if we take one partition with thin components (i.e., nearly every vertex is on the boundary of a component) and construct a new partition by slightly perturbing the first. We expect little overlap between these two partitions, but this small perturbation can be corrected by moving mass a short distance, which gives rise to a high $L_1$ distance but a low transport distance. 
We give a concrete example in \Cref{sec:grid_ham_ot}, where two partitions of the grid graph with ``snakey'' components are far apart in $L_1$ but not in transport distance because the first partition roughly looks like a small perturbation of the second.

\section{Experiments and Empirical Evaluation}\label{sec:experiments}

In this section, we provide a suite of empirical applications of our metric using both synthetic data from grid partitions and real geographic data including election results.

\subsection{Implementation and Experimental Setup}

We run all of the following experiments using the Python programming language on consumer-grade hardware, and our code is available on GitHub.\footnote{\url{https://github.com/vrdi/geometry-of-graph-partitions}}  We use the packages CVXPY \cite{cvxpy} and \texttt{scikit-learn} \cite{scikit-learn} to perform the optimization and embedding. The grid graph partitions in \cref{sec:grid_parts} and the Markov chain sampling in
\cref{sec:arkansas,sec:iowa,sec:nc_ensemble} rely on the \texttt{enumerator}
\cite{schutzman2019enumerator} and \texttt{GerryChain} \cite{gerrychain} open source software packages 
available on GitHub.  The spatial and electoral data comes from the Metric Geometry and Gerrymandering Group's \texttt{mggg-states}
repository \cite{mggg_states}, the NHGIS database \cite{manson2017ipums}, and the data
accompanying \cite{Mattingly2018} from the Quantifying Gerrymandering group \cite{nc_data}. 

Many of our examples embed multiple partitions of a fixed graph onto the plane to visualize an ensemble, using the method outlined below. 
We store the pairwise distances between the $n$ partitions in a matrix $D$, where $D_{ij}$ is the distance between partitions $i$ and $j$.  We then would like to find a set of points $P_1,\ldots,P_n\in\R^2$ such that the Euclidean distance between $P_i$ and $P_j$ is equal to $D_{ij}$.  Doing so with zero distortion may be impossible or may require using a high-dimensional ambient space that is impossible to visualize meaningfully.  To resolve this issue, we use \textit{multidimensional scaling} (MDS), which computes $P_1,\ldots,P_n$ 
in a way that (approximately) minimizes the sum of the squared residuals $\left(D_{ij} - \mathrm{dist}(P_i,P_j) \right)^2$.  For a modern treatment of MDS including algorithms and applications, see \cite{borg2012applied}.

\subsection{Grid Partitions}\label{sec:grid_parts}

We begin by examining the distance between partitions of a \textit{grid graph}.  While the number of feasible districting plans for a US state is unfathomably large, for small grids, the number of partitions is much more manageable; for example, there are only 117 ways to partition a $4\times4$ grid into four connected components with four vertices in each component.  We can therefore compute the transport distance between all or a large portion of the possible partitions and embed them in the plane, gaining an intuitive understanding of the qualitative similarities of nearby partitions. 

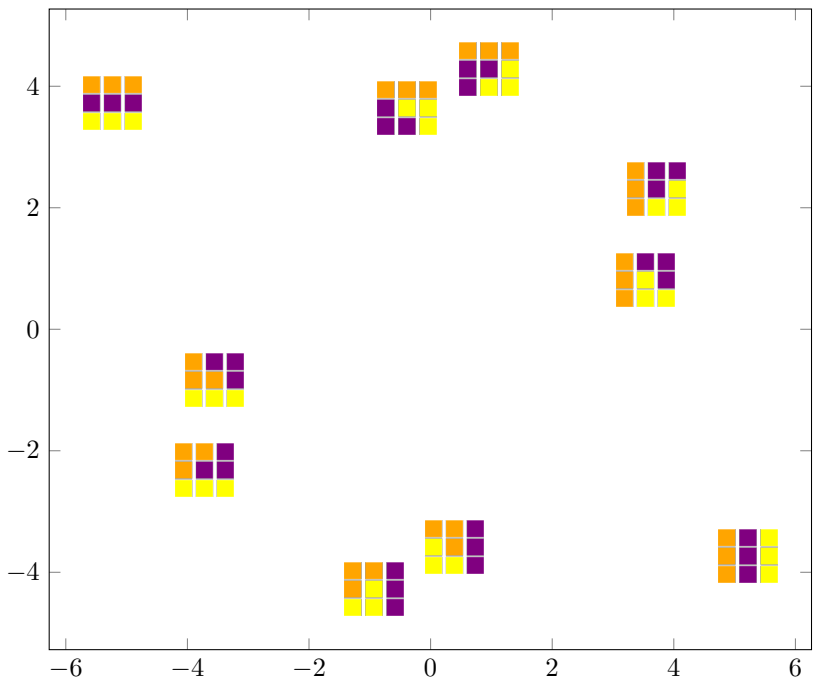
\begin{figure} 
\centering 
\begin{tikzpicture} 
\begin{axis}[width=.9\textwidth, axis equal] 

\pgfplotsinvokeforeach{0,...,9}{ 
            \addplot[ 
            text mark={\includegraphics[width=22pt]{experiments/grid/3-by-3-imgs/#1.png}}, 
            mark=text, 
            only marks, 
            ] table[ 
            x = x, 
            y=y, 
            col sep=comma,] {experiments/grid/3-by-3-csvs/3_3-3__#1.csv}; 
            };
\end{axis}
\end{tikzpicture}
\caption{The ten partitions of the $3\times3$ grid into three equal-sized components.}\label{fig:3x3grid}
\end{figure}

In \cref{fig:3x3grid}, we show the MDS embedding of the pairwise transport distances between the ten partitions of a $3\times3$ grid into three connected components of size three (i.e., the \textit{triomino tilings} of the $3\times3$ grid).  This visualization reveals several features of the metric space.  The two partitions into horizontal and vertical `stripes' are the furthest apart, and the remaining eight partitions cluster in pairs based on which straight triomino is included in the partition.  For example, the two partitions near the middle-top both have a straight triomino along the top row and two `L'-triominos covering the lower two rows.  Furthermore, the partitions that include a horizontal straight triomino fall in the top-left half and the ones with a vertical straight triomino fall in the lower-right half.  The two partitions closest to the top-right corner share an 'L'-triomino, as does the pair in the lower-left.

\begin{figure} 
\centering 
\begin{tikzpicture} 
\begin{axis}[width=.9\textwidth, axis equal] 

\pgfplotsinvokeforeach{0,...,116}{ 
            \addplot[ 
            text mark={\includegraphics[width=12pt]{experiments/grid/4-by-4-imgs/#1.png}}, 
            mark=text, 
            only marks, 
            ] table[ 
            x = x, 
            y=y, 
            col sep=comma,] {experiments/grid/4-by-4-csvs/4_4-4__#1.csv}; 
            };

\end{axis}
\end{tikzpicture}
\caption{The 117 partitions of the $4\times4$ grid into four equal-sized components.}\label{fig:4by4}
\end{figure}
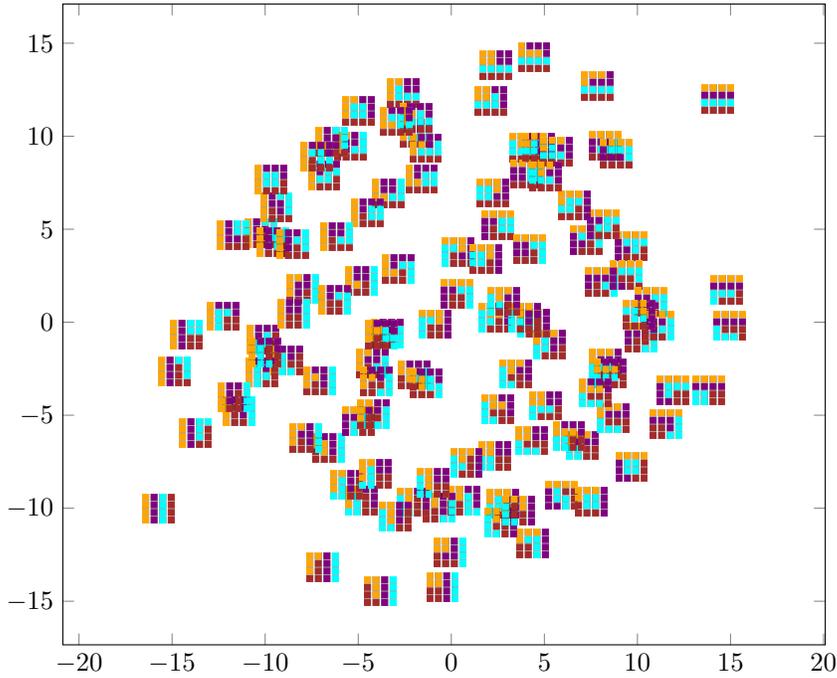

Similar phenomena can be observed in the embedding of the 117 partitions of the $4\times4$ grid into four connected components of size four (i.e. the \textit{tetromino tilings}) in \cref{fig:4by4}. Again, the two `striped' partitions are the furthest apart, and there are visible clusters of partitions with similar compositions.  Additionally, there are two partitions each composed of four `T'-tetrominos and two partitions composed of four `L'-tetrominos which are at the center of this metric space, and these appear in the middle of the image.  \Cref{fig:4by4-reduced} highlights this structure.

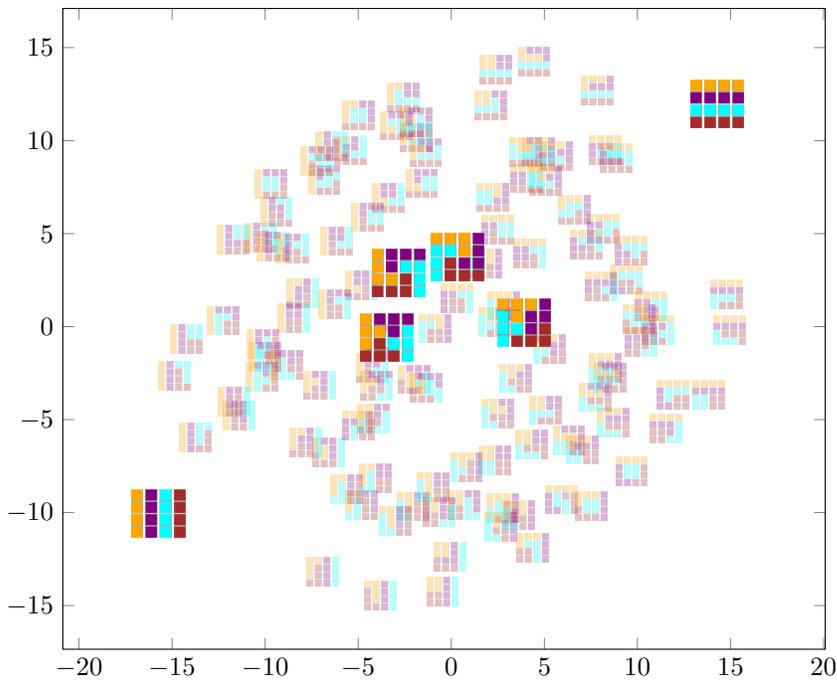
\begin{figure} 
\centering 
\begin{tikzpicture} 
\begin{axis}[width=.9\textwidth, axis equal] 

\pgfplotsinvokeforeach{0,...,19}{ .
            \addplot[
            text mark={\includegraphics[width=12pt]{experiments/grid/4-by-4-imgs/#1.png}}, 
            mark=text, 
            only marks,
            opacity=.3,
            ] table[     
            x = x, 
            y=y, 
            col sep=comma,] {experiments/grid/4-by-4-csvs/4_4-4__#1.csv}; 
            };

\pgfplotsinvokeforeach{21,...,41}{ 
            \addplot[ 
            text mark={\includegraphics[width=12pt]{experiments/grid/4-by-4-imgs/#1.png}}, 
            mark=text,
            only marks,
            opacity=.3, 
            ] table[          
            x = x, 
            y=y,
            col sep=comma,] {experiments/grid/4-by-4-csvs/4_4-4__#1.csv}; 
            };

\pgfplotsinvokeforeach{43,...,71}{ 
            \addplot[ 
            text mark={\includegraphics[width=12pt]{experiments/grid/4-by-4-imgs/#1.png}}, 
            mark=text,
            only marks,
            opacity=.3,
            ] table[          
            x = x, 
            y=y,
            col sep=comma,] {experiments/grid/4-by-4-csvs/4_4-4__#1.csv}; 
            };

\pgfplotsinvokeforeach{73,...,81}{ 
            \addplot[ 
            text mark={\includegraphics[width=12pt]{experiments/grid/4-by-4-imgs/#1.png}}, 
            mark=text,
            only marks,
            opacity=.3,
            ] table[          
            x = x, 
            y=y,
            col sep=comma,] {experiments/grid/4-by-4-csvs/4_4-4__#1.csv}; 
            };            
            
\pgfplotsinvokeforeach{83,...,95}{ 
            \addplot[ 
            text mark={\includegraphics[width=12pt]{experiments/grid/4-by-4-imgs/#1.png}}, 
            mark=text,
            only marks,
            opacity=.3,
            ] table[          
            x = x, 
            y=y,
            col sep=comma,] {experiments/grid/4-by-4-csvs/4_4-4__#1.csv}; 
            };
            
\pgfplotsinvokeforeach{97,...,110}{ 
            \addplot[ 
            text mark={\includegraphics[width=12pt]{experiments/grid/4-by-4-imgs/#1.png}}, 
            mark=text,
            only marks,
            opacity=.3,
            ] table[          
            x = x, 
            y=y,
            col sep=comma,] {experiments/grid/4-by-4-csvs/4_4-4__#1.csv}; 
            };           
            
\pgfplotsinvokeforeach{112,...,116}{ 
            \addplot[ 
            text mark={\includegraphics[width=12pt]{experiments/grid/4-by-4-imgs/#1.png}}, 
            mark=text,
            only marks,
            opacity=.3,
            ] table[          
            x = x, 
            y=y,
            col sep=comma,] {experiments/grid/4-by-4-csvs/4_4-4__#1.csv}; 
            };           
            
\pgfplotsinvokeforeach{20,42,72,82,96,111}{ 
            \addplot[ 
            text mark={\includegraphics[width=20pt]{experiments/grid/4-by-4-imgs/#1.png}}, 
            mark=text,
            only marks,
            ] table[          
            x = x, 
            y=y,
            col sep=comma,] {experiments/grid/4-by-4-csvs/4_4-4__#1.csv}; 
            };                      

\end{axis}
\end{tikzpicture}
\caption{The four central and two peripheral partitions of the $4\times4$ grid.}\label{fig:4by4-reduced}
\end{figure}

As a illustration of the balanced case, in \cref{fig:6by6-sample} we sample from the partitions of the $6\times6$ grid into three equal-sized components.  There are 264,500 such partitions and hence computing the pairwise distances and embedding all of them is not practical or informative.  Rather, we examine a random sample of 100 partitions, plus the two `striped' partitions.  The familiar structure emerges: the `striped' partitions are the furthest apart and visually similar partitions appear near one another in the embedding.

\begin{figure}
\centering 
\begin{tikzpicture}[scale=.45] 
\begin{axis}[width=\textwidth, axis equal] 

\pgfplotsinvokeforeach{0,...,101}{ 
            \addplot[ 
            text mark={\includegraphics[width=14pt]{experiments/grid/6-by-6-imgs/sample1/#1.png}}, 
            mark=text,
            only marks, 
            ] table[           
            x = x, 
            y=y, 
            col sep=comma,] {experiments/grid/6-by-6-csvs/sample1/6_6-3__#1.csv}; 
            };

\end{axis}
\end{tikzpicture}
\qquad
\begin{tikzpicture}[scale=.45]
\begin{axis}[width=\textwidth, axis equal] 
\pgfplotsinvokeforeach{0,...,101}{ 
            \addplot[ 
            text mark={\includegraphics[width=14pt]{experiments/grid/6-by-6-imgs/sample7/#1.png}}, 
            mark=text,
            only marks, 
            ] table[           
            x = x, 
            y=y, 
            col sep=comma,] {experiments/grid/6-by-6-csvs/sample7/6_6-3__#1.csv}; 
            };

\end{axis}
\end{tikzpicture}
\caption{Two samples of 102 partitions of the $6\times6$ grid into three components.}\label{fig:6by6-sample}
\end{figure}
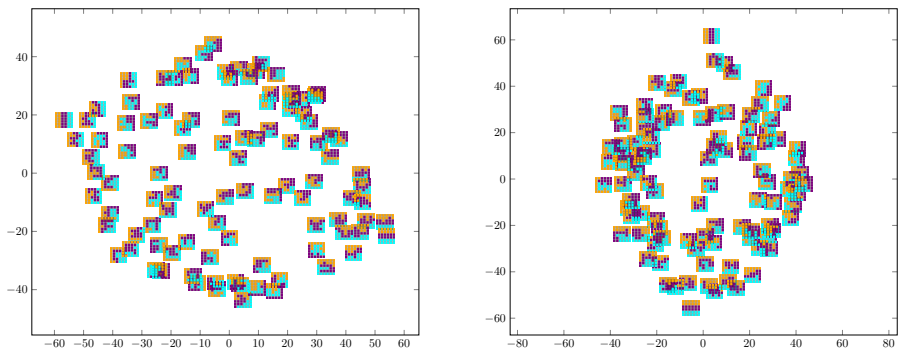

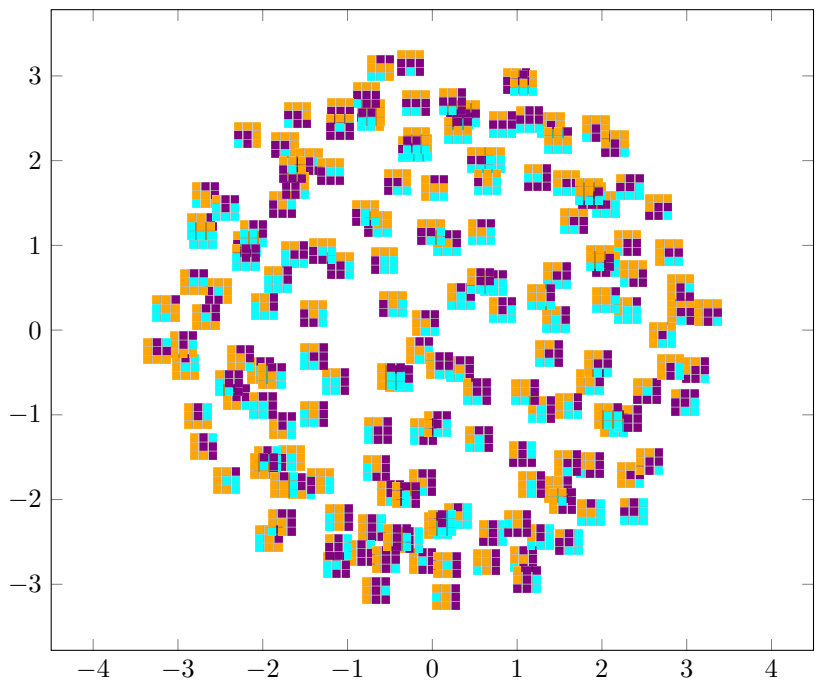
\begin{figure}
\centering
\begin{tikzpicture}
\begin{axis}[width=.9\textwidth, axis equal] 
\pgfplotsinvokeforeach{0,...,169}{ 
            \addplot[ 
            text mark={\includegraphics[width=10pt]{experiments/grid/3-by-3-pm2-imgs/#1.png}}, 
            mark=text,
            only marks, 
            ] table[           
            x = x, 
            y=y, 
            col sep=comma,] {experiments/grid/3-by-3-pm2-csvs/pen_05/#1.csv}; 
            };
\end{axis}
\end{tikzpicture}
\caption{The partitions of the $3\times 3$ grid into three components of size $3\pm 2$ with a penalty of $\lambda=.5$}\label{fig:3by3-pm2-05}
\end{figure}

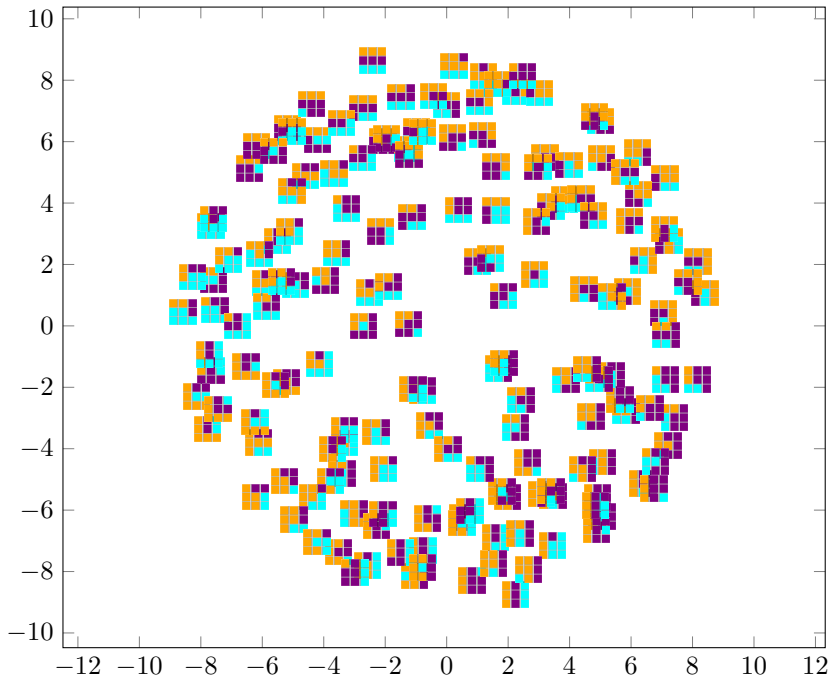
\begin{figure}
\centering
\begin{tikzpicture}
\begin{axis}[width=.9\textwidth, axis equal] 
\pgfplotsinvokeforeach{0,...,169}{ 
            \addplot[ 
            text mark={\includegraphics[width=10pt]{experiments/grid/3-by-3-pm2-imgs/#1.png}}, 
            mark=text,
            only marks, 
            ] table[           
            x = x, 
            y=y, 
            col sep=comma,] {experiments/grid/3-by-3-pm2-csvs/pen_2/#1.csv}; 
            };
\end{axis}
\end{tikzpicture}
\caption{The partitions of the $3\times 3$ grid into three components of size $3\pm 2$ with a penalty of $\lambda=2$}\label{fig:3by3-pm2-2}
\end{figure}

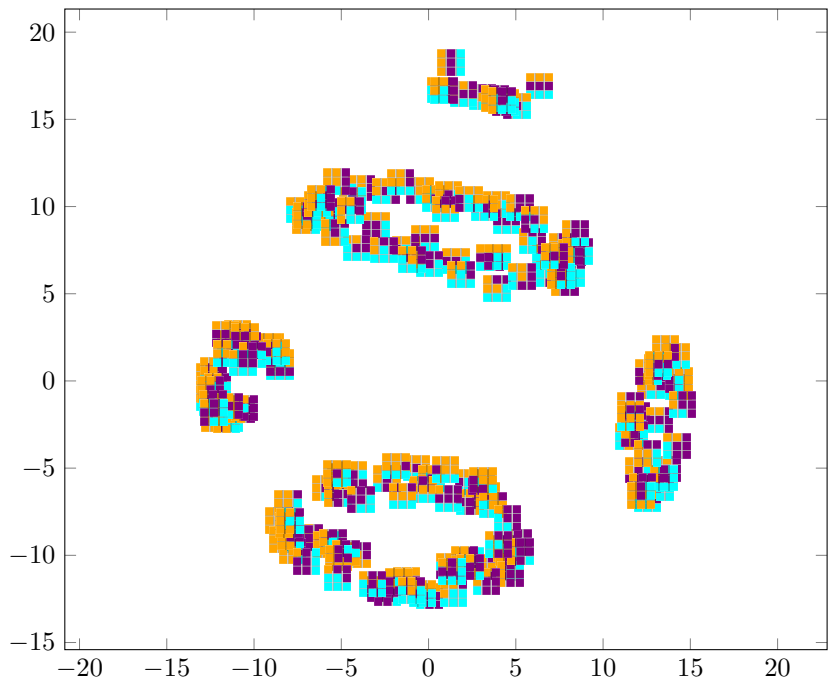
\begin{figure}
\centering
\begin{tikzpicture}
\begin{axis}[width=.9\textwidth, axis equal] 
\pgfplotsinvokeforeach{0,...,169}{ 
            \addplot[ 
            text mark={\includegraphics[width=10pt]{experiments/grid/3-by-3-pm2-imgs/#1.png}}, 
            mark=text,
            only marks, 
            ] table[           
            x = x, 
            y=y, 
            col sep=comma,] {experiments/grid/3-by-3-pm2-csvs/pen_5/#1.csv}; 
            };
\end{axis}
\end{tikzpicture}
\caption{The partitions of the $3\times 3$ grid into three components of size $3\pm 2$ with a penalty of $\lambda=5$}\label{fig:3by3-pm2-5}
\end{figure}

To examine the \textit{unbalanced} problem, we look at partitions of the $3\times3$ grid into three components of size between one and five, inclusive.  There are 170 such partitions, including those in \cref{fig:3x3grid}.  We use the distance in \cref{prop:unbalanced_simplified_lp} using the 1-norm on $z$ in the objective.  We vary $\lambda$  to illustrate properties of the unbalanced cost function for different parameter regimes:
\begin{itemize}
\item First, we take $\lambda=0.5$, so the cost of sending one unit of flow through an edge is the same as leaving that unit as unbalanced mass in the variable $z$.
This formulation is similar to both the Hamming and total variation distance (see \cref{section:bounds}), where the cost to match component $x_i$ in the first partition to component $y_j$ in the second is the number of vertices in the set difference $x_i\setminus y_j$.  The resulting embedding in \cref{fig:3by3-pm2-05} confirms this
set difference observation: Nearby partitions tend to have a component in common and the other two component are similar to each other, often differing in the assignment of only one vertex.
\item Next, we consider $\lambda=2$, where  the cost of leaving one unit in $z$ equals the graph diameter. As in \cref{prop:large_lambda}, for any pairing of the partition components, it will be suboptimal to move mass through $z$ unless absolutely necessary, in the case that two components of differing mass are matched.  
\item The $\lambda=2$ case contrasts significantly with $\lambda=5$, shown in \cref{fig:3by3-pm2-5}. Here, the cost of having mass in $z$
is so high that the distance between any pair of unbalanced partitions should be larger than the distance between any pair of balanced partitions.  In the embedding, we therefore see the partitions for which there exists a matching of components of equal mass cluster together.  Furthermore, we see the familiar structure emerge in each cluster separately.  For example, the ten partitions consisting of three components of size three appear together at the top of the \cref{fig:3by3-pm2-5} in an arrangement similar to the one in \cref{fig:3x3grid}.
\end{itemize}

\subsection{Hamming Distance vs.\ Transport Distance on a Grid}
\label{sec:grid_ham_ot}

Because the Hamming distance depends only on the amount of overlap between two districts in a partition, it is insensitive to the distances between vertices in the graph. To illustrate the difference between the Hamming and transport distances, we run two Markov chains of partitions of a $60\times60$ grid into six components. 

These Markov chains, at each time step, relabel a random vertex in the graph. If the resultant partition consists of connected components that all have nearly the same number of vertices, this partition is accepted and we propose another vertex to relabel. Otherwise, the step is rejected and we retry.  As shown in \cite{najt2019complexity}, this procedure results in partitions with geometrically irregular and highly non-compact components; as an example, the start and end positions of each chain are shown in \cref{fig:startandendongrids}. 

We let both Markov chains run and compute the distance between them at every 1,000th step, using Hamming and transport distance; \cref{fig:tipdistances} shows the result. Over time, the Hamming distance increases while the transport distance decreases. The MDS plots in \cref{fig:mdsflipwalks} show that in this situation, the transport distance and the Hamming distance distinguish qualitatively different properties.  The Hamming distance highlights that the two initial partitions have a significant amount of overlapping vertices and the final partitions do not.  As a point of contrast, the transport distance detects that transforming the initial partition into the other requires moving a significant amount of mass a large distance through the graph, while much less work is needed to match two final plans since their boundaries are interleaved. Put differently, partitions whose boundaries are long and intertwined are considered similar under the transport metric because mass does not have to be displaced a long distance to convert one into the other, while Hamming distances simply count overlapping vertices.

\begin{figure}
\centering
\begin{subfigure}{0.24\textwidth}
\centering
    \includegraphics[width=\textwidth]{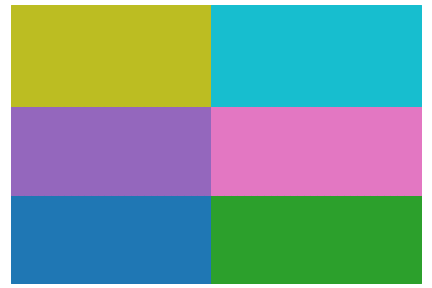}
    \subcaption{Chain 1 start}
\end{subfigure}
\begin{subfigure}{0.24\textwidth}
\centering
    \includegraphics[width=\textwidth]{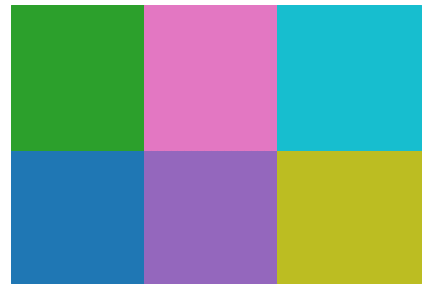}
    \subcaption{Chain 2 start}
\end{subfigure}
\begin{subfigure}{0.24\textwidth}
\centering
    \includegraphics[width=\textwidth]{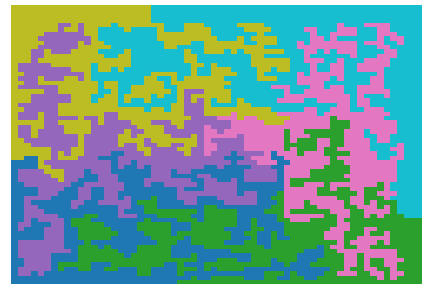}
    \subcaption{Chain 1 end}
\end{subfigure}
\begin{subfigure}{0.24\textwidth}
\centering
    \includegraphics[width=\textwidth]{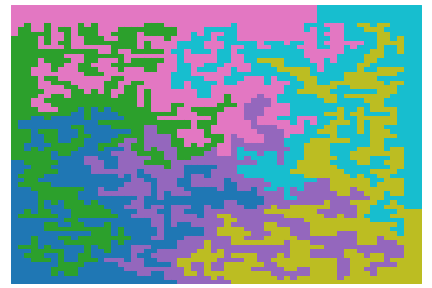}
    \subcaption{Chain 2 end}
\end{subfigure}

\caption{Two Markov chains on a 60-by-60 grid}
\label{fig:startandendongrids}
\end{figure}

\begin{figure}
\begin{subfigure}{0.5\textwidth}
\centering
    \resizebox{\columnwidth}{!}{%
   \begin{tikzpicture}
	\begin{axis}[
    xlabel={Step number},
    ylabel={Distance apart}]
	\addplot table[x=step, y=d, col sep=comma, mark=none]{experiments/flipwalks/tipsH.csv};
	\end{axis}
	\end{tikzpicture}
	}
    \caption{Hamming distance}
\end{subfigure}
\begin{subfigure}{0.47\textwidth}
\centering
    \resizebox{\columnwidth}{!}{%
    \begin{tikzpicture}
	\begin{axis}[
    xlabel={Step number},
    ]
    \addplot table[x=step, y=d, col sep=comma, mark=none]{experiments/flipwalks/tipsT.csv};
	\end{axis}    
	\end{tikzpicture}
	}
	\caption{Transport distance}
\end{subfigure}

\caption{Distance between Chain 1 and Chain 2}
\label{fig:tipdistances}
\end{figure}
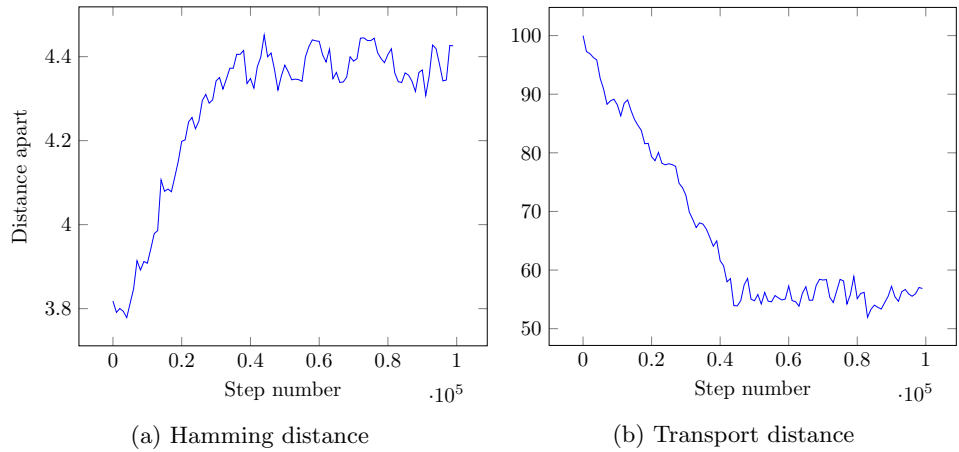

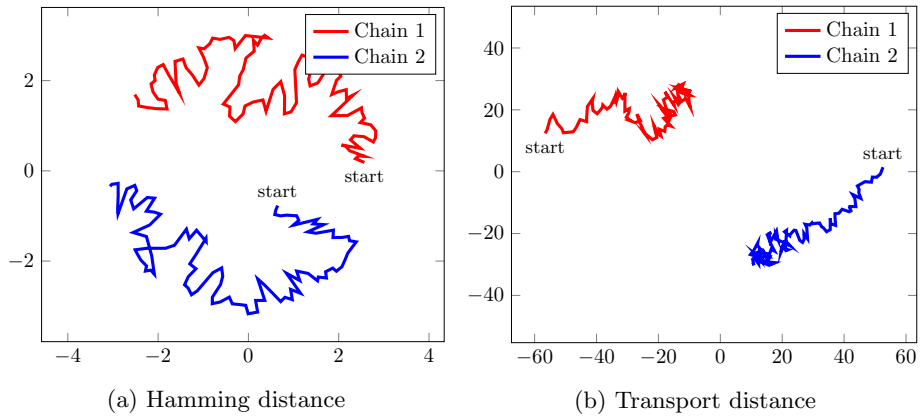
\begin{figure}
\centering
\begin{subfigure}{0.46\textwidth}
\centering
    \resizebox{\columnwidth}{!}{%
   \begin{tikzpicture}
	\begin{axis}[scatter/classes={a={red},b={blue}}, axis equal]
	\addplot[red, mark=none, line width=1.5pt] table[x=X, y=Y, col sep=comma]{experiments/flipwalks/flipwalkHXY1.csv}
	node[pos=0, below, text=black]{\small start};
	\addplot[blue, mark=none, line width=1.5pt] table[x=X, y=Y, col sep=comma]{experiments/flipwalks/flipwalkHXY2.csv}
	node[pos=0, above, text=black]{\small start};
	\legend{Chain 1,Chain 2}
	\end{axis}
	\end{tikzpicture}
	}
    \caption{Hamming distance}
\end{subfigure}
\begin{subfigure}{0.48\textwidth}
\centering
    \resizebox{\columnwidth}{!}{%
    \begin{tikzpicture}
	\begin{axis}[axis equal]
	\addplot[red, mark=none, line width=1.5pt] table[x=X, y=Y, col sep=comma]{experiments/flipwalks/flipwalkTXY1.csv}
	node[pos=0, below, text=black]{\small start};
	\addplot[blue, mark=none, line width=1.5pt] table[x=X, y=Y, col sep=comma]{experiments/flipwalks/flipwalkTXY2.csv}
	node[pos=0, above, text=black]{\small start};
	\legend{Chain 1,Chain 2}
	\end{axis}
	\end{tikzpicture}
	}
	\caption{Transport distance}
\end{subfigure}

\caption{MDS embeddings of Chain 1 and Chain 2}
\label{fig:mdsflipwalks}
\end{figure}

\subsection{Simulated annealing on Arkansas}\label{sec:arkansas}
A more sophisticated technique for generating partitions meeting a specific criterion uses simulated annealing, where a weighting function is tuned over time to first allow rapid exploration of the space of possibilities and later to settle into a local optimum. 
To show how our distance can be used to analyze the behavior of such a process, we run a Markov chain to partition Arkansas into four congressional districts, focusing on producing compact plans. We take 500,000 steps along the Markov chain using the random relabelling proposal as in \cref{sec:grid_ham_ot}.  To perform the annealing, the first 100,000 are taken without weighting, for steps 100,000 to 500,000 we accept a proposed step with probability proportional to $\exp(\beta |\partial P|)$, where $\partial P$ denotes the number of edges in the graph which join two vertices in different components, which represents a discretization of the \textit{boundary length} of the districts in the plan. The parameter $\beta$ ranges linearly from $\beta = 0$ (no weighting) to $\beta = 3$ from step 100,000 to step 400,000, and is kept fixed at $3$ for the final 100,000 steps.

\begin{figure}
\centering
\begin{tabular}{@{}c@{}c@{}c@{}c@{}}
\includegraphics[width=.25\textwidth]{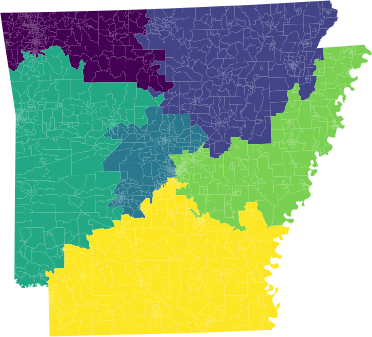}&
\includegraphics[width=.25\textwidth]{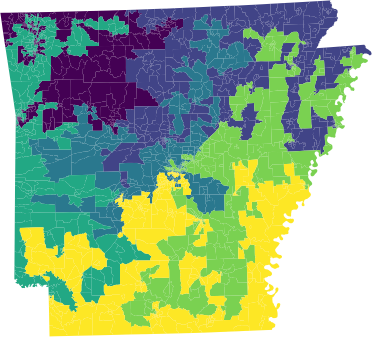}&
\includegraphics[width=.25\textwidth]{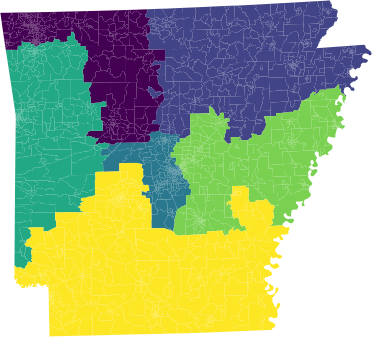}&
\includegraphics[width=.25\textwidth]{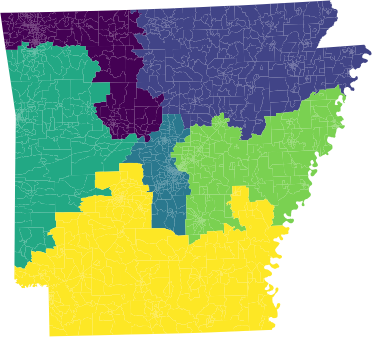}\\
Initial & 100,000 steps & 400,000 steps & Final
\end{tabular}

    \caption{Snapshots of a simulated annealing Markov chain on Arkansas districts.}
\label{fig:arkansas-annealing}
\end{figure}

An embedding of every 10,000th plan in the chain is shown in \cref{fig:annealXY}. The plot provides an effective visualization showing that when the chain is unrestrained by the Metropolis weighting, it moves more quickly through the state space. Also, the chain initially moves away from the starting plan before being brought closer to the initial plan once the Metropolis weighting is introduced. This observation can be confirmed in the snapshots in \cref{fig:arkansas-annealing}.

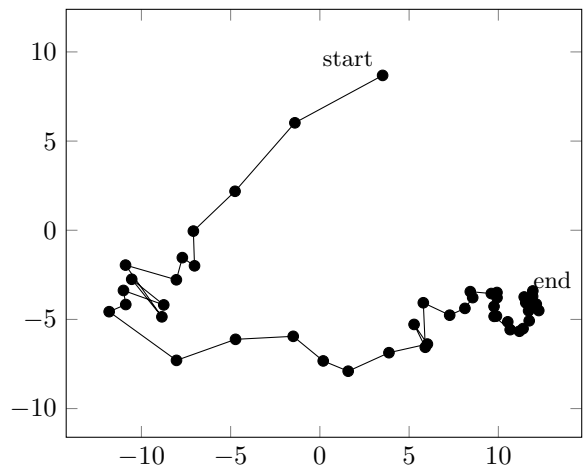
\begin{figure}
\centering
\begin{tikzpicture}[scale=1.]
\begin{axis}[axis equal]

\addplot[color=black, mark=*] table[x=X, y=Y, col sep=comma]{experiments/annealing/annealXY.csv}
node[pos=0, above left, text=black]{\small start}
node[pos=1, above right, text=black]{\small end};
\end{axis}
\end{tikzpicture}
\caption{A Markov chain-generated walk in the space of partitions with simulated annealing.}
\label{fig:annealXY}
\end{figure}

\subsection{Partisan clustering}\label{sec:iowa}
Given a large ensemble of districting plans, we can investigate the geographic features of those plans with extreme partisan statistics. To illustrate how our distance can be used for such an analysis, we consider an ensemble of congressional districting plans for Iowa, generated by a `recombination' Markov chain \cite{deford2019recom}. Here, rather than choosing a single random vertex to relabel, we instead randomly choose two components of the partition to merge and re-split into two new components.  

\Cref{fig:iowa08,fig:iowa12} show embeddings of these plans, colored by the number of seats won under two different historical election results: the 2008 Presidential election, in which the Democratic candidate Barack Obama won approximately 55 percent of the two-way vote share against Republican John McCain, and the 2012 Presidential election, in which Barack Obama won approximately 53 percent of the two-way vote share against Republican Mitt Romney. In each case, most plans produce three Democratic seats but a small number do not, and these tend to cluster near one another, with a stronger pattern for the 2012 voting data. \Cref{fig:iowamaps} shows the maps of each of the plans which, under the 2012 data, had two majority Democratic districts, which shows the geographic similarity which resulted in the clustering in the embedding.

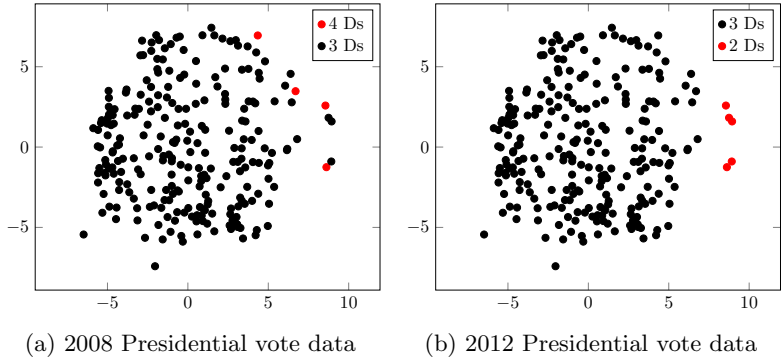
\begin{figure}
\centering
\begin{subfigure}{0.4\textwidth}
\centering
\resizebox{\columnwidth}{!}{
\begin{tikzpicture}
	\begin{axis}[axis equal]
	\addplot[color=red, mark=*, only marks] table[x=X, y=Y, col sep=comma]{experiments/iowa/four_seats_2008.csv};
	\addlegendentry{4 Ds}
	\addplot[color=black, mark=*, only marks] table[x=X, y=Y, col sep=comma]{experiments/iowa/three_seats_2008.csv};
	\addlegendentry{3 Ds}
	\end{axis}
\end{tikzpicture}
}
\caption{2008 Presidential vote data}
\label{fig:iowa08}
\end{subfigure}
\begin{subfigure}{0.4\textwidth}
\centering
\resizebox{\columnwidth}{!}{
\begin{tikzpicture}
	\begin{axis}[axis equal]
	\addplot[color=black, mark=*, only marks] table[x=X, y=Y, col sep=comma]{experiments/iowa/three_seats_2012.csv};
	\addlegendentry{3 Ds}
	\addplot[color=red, mark=*, only marks] table[x=X, y=Y, col sep=comma]{experiments/iowa/two_seats_2012.csv};
	\addlegendentry{2 Ds}
	\end{axis}
\end{tikzpicture}
}
\caption{2012 Presidential vote data}
\label{fig:iowa12}
\end{subfigure}
\caption{Iowa ensemble colored by Democratic seats won.}
\end{figure}

\begin{figure}
\centering
\begin{tabular}{@{}c@{}c@{}c@{}c@{}c@{}}
\includegraphics[width=.2\textwidth]{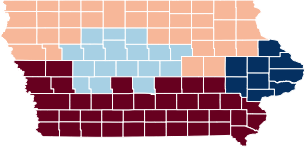}&
\includegraphics[width=.2\textwidth]{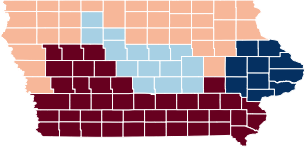}&
\includegraphics[width=.2\textwidth]{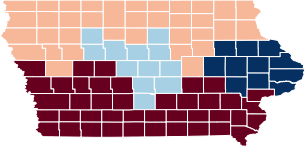}&
\includegraphics[width=.2\textwidth]{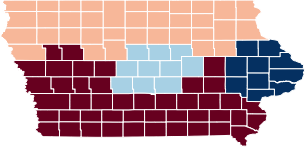}&
\includegraphics[width=.2\textwidth]{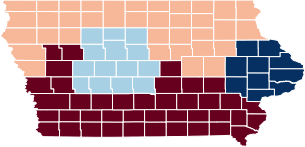}
\end{tabular}
    \caption{Partitions of Iowa with two Democratic seats under 2012 Presidential election results.  The two blue districts have a Democratic majority and the two red districts have a Republican majority.}
\label{fig:iowamaps}
\end{figure}

\subsection{Outlier analysis}\label{sec:nc_ensemble}
In \cite{Mattingly2018}, the authors demonstrate that the congressional districting plans enacted in 2012 and in 2016 were atypical outliers in terms of their partisan statistics (computed using fixed vote data from multiple elections), relative to a computer-generated ensemble of plans. A third plan proposed by a bipartisan panel of judges, on the other hand, was found to be far more representative. \Cref{fig:MattinglyMDS} shows an MDS embedding of one hundred plans drawn from the the authors' ensemble as well as the three human-drawn plans mentioned above, using our distance. In \cref{fig:ReComMDS}, we show the same plot but with an ensemble generated by a `recombination' Markov chain.

The judges' plan lies near the middle of the ensemble in both cases, whereas the 2012 and 2016 plans lie on the edge. This indicates that the judges' plan is far more representative of the ensemble than the 2012 and 2016 plan in terms of its geography. The authors in \cite{Mattingly2018} show that the judges' plan is \textit{politically} representative of the ensemble and here we see that it is \textit{geographically} representative as well. 

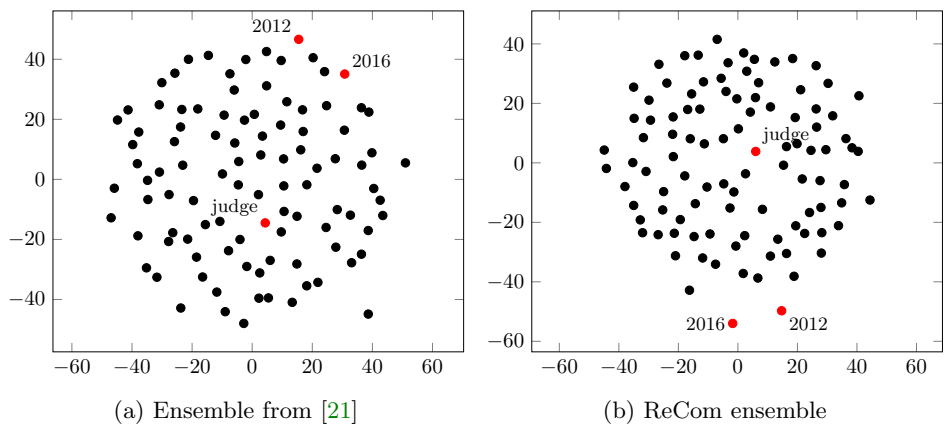
\begin{figure}
\begin{subfigure}{0.48\textwidth}
\centering
\resizebox{\columnwidth}{!}{%
\begin{tikzpicture}[scale=1]
\begin{axis}[scatter/classes={a={red},b={black}}, axis equal]
\addplot[scatter, only marks, scatter src=explicit symbolic] table[meta=label, x=X, y=Y, col sep=comma]{experiments/Mattingly/MattinglyMDS.csv}
node[pos=0, above left, text=black]{\small judge}
node[pos=0.01, above left, text=black]{\small 2012}
node[pos=0.02, above right, text=black]{\small 2016};
\end{axis}
\end{tikzpicture}
}
\caption{Ensemble from \cite{Mattingly2018}}
\label{fig:MattinglyMDS}
\end{subfigure}
\begin{subfigure}{0.48\textwidth}
\centering
\resizebox{\columnwidth}{!}{%
\begin{tikzpicture}[scale=1]
\begin{axis}[scatter/classes={a={red},b={black}}, axis equal]
\addplot[scatter, only marks, scatter src=explicit symbolic] table[meta=label, x=X, y=Y, col sep=comma]{experiments/Mattingly/RecomXY.csv}
node[pos=0, above right, text=black]{\small judge}
node[pos=0.01, below right, text=black]{\small 2012}
node[pos=0.02, left, text=black]{\small 2016};
\end{axis}
\end{tikzpicture}
}
\caption{ReCom ensemble}
\label{fig:ReComMDS}
\end{subfigure}
\caption{Transport embeddings of ensembles and human-drawn plans for North Carolina.}
\end{figure}

\section{Conclusion}\label{sec:conclusion}

Through a straightforward construction, we demonstrate 
how optimal transport---already a lifting of a geometric structure to the set of probability measures---can be further lifted to a geometry on partitions.  Our definition as a transport problem whose cost function is itself the result of solving a transport problem is an intriguing example of \emph{hierarchical} transport in its own right. We demonstrate that several intuitve and classical results about transport and network flow apply in this lifted setting. Moreover, by restricting this distance to the space of partitions rather than general (unordered) collections of measures, we are able to derive some specialized results.  

Relatively few works have put a geometry on the space of partitions, and our progress on this problem suggests several avenues for future research.  We mention a few below:
\begin{itemize}
    \item A well-known challenge in redistricting is the sheer number of ways to partition a graph, which obstructs global analysis of a districting plan relative to all possible alternatives.  By putting a geometry on the space of partitions, we can ask whether the combinatorial count of partitions is truly insurmountable, or if there is a ``small neighborhood'' phenomenon whereby most partitions are close to a relatively small spanning subset of representatives.
    \item While our construction puts an intuitive and interpretable distance on the space of partitions, it can be expensive to compute relative to the more naive Hamming distances and total variation distances (see \cref{section:bounds}), which have far weaker geometric behavior but are often very easy to compute.  It may be the case that alternatives exist with a better compromise between computational efficiency and expressiveness.
    \item Motivated by our intended applications in redistricting, the constructions in this paper are discrete and restricted to the 1-Wasserstein distance.  Our definition readily generalizes to partitions of compact regions in $\R^n$, although the underlying computational problem becomes much more challenging.
\end{itemize}

Ultimately, as demonstrated in \cref{sec:experiments}, our distance is not only a valuable mathematical construction but also---perhaps more importantly---a practical tool needed in emerging applications of data analysis to political science.  Equipped with our distance and embedding algorithms, we can quickly navigate and judge the extent of a collection of partitions, addressing a significant gap in current methodologies for ensemble-based redistricting.

\section*{Acknowledgments} The authors acknowledge the generous support and collaboration of several colleagues.  In addition to the authors, the original team at the Voting Rights Data Institute (VRDI) working on distances between districting plans included Kristen Akey, Sonali Durham, Maira Khan, Jasmine Noory, Gabe Schoenbach, and M\'elisande Teng.  We additionally thank Ruth Buck, Sebastian Claici, Daryl DeFord, Moon Duchin, Lorenzo Najt, David Palmer, and Paul Zhang for valuable discussions throughout the research process.  J.\ Solomon acknowledges the generous support of NSF grant IIS-1838071,  Air Force Office of Scientific Research award FA9550-19-1-0319, Army Research Office grant W911NF-12-R-001, and the Prof.\ Amar G.\ Bose Research Grant.  N.\ Guillen acknowledges the generous support of NSF Grant DMS-1700307. Any opinions, findings, and conclusions or recommendations expressed in this material are those of the authors and do not necessarily reflect the views of these organizations.

\bibliographystyle{siamplain}
\bibliography{references}

\end{document}